\documentclass[12pt]{article}
\usepackage{amssymb,oldlfont}
\usepackage{amsmath}
\usepackage{hyperref}
\usepackage{color}
\definecolor{myorange}{RGB}{180,90,0}
\definecolor{mygreen}{RGB}{70,140,0}
\usepackage[normalem]{ulem}

\usepackage[T1]{fontenc}
\usepackage[utf8]{inputenc}
\def\wrtext#1{\relax\ifmmode{\leavevmode\hbox{#1}}\else{#1}\fi}
\def\abs#1{\left|#1\right|}
\def\begeq{\begin{equation}}
\def\endeq{\end{equation}}

\def\part#1{\frac{\partial}{\partial #1}}

\def\norm#1{||\,#1\,||}

\newcommand{\real}{\mbox{\bf R}}
\newcommand{\comp}{\mbox{\bf C}}

\newcommand{\nat}{\mbox{\bf N}}

\renewcommand{\exp}{\mbox{\rm exp\,}}

\newtheorem{dref}{Definition}[section]
\newtheorem{lemma}[dref]{Lemma}
\newtheorem{theo}[dref]{Theorem}
\newtheorem{prop}[dref]{Proposition}

\newenvironment{proof}{\vspace{.3cm}\noindent{{\em Proof:}}}{\hfill$\Box$}

\title{Complex FIOs and composition of Toeplitz operators}
\author{Lewis \textsc{Coburn} \footnote{Department of Mathematics, SUNY at Buffalo, Buffalo, NY 14260, USA, {\sf lcoburn@buffalo.edu}}\and Michael \textsc{Hitrik} \footnote{Department of Mathematics, UCLA, Los Angeles CA 90095-1555, USA, {\sf hitrik@math.ucla.edu}} \and \and Johannes \textsc{Sj\"ostrand}\footnote{IMB, Universit\'e de Bourgogne 9, Av. A. Savary, BP 47870
FR-21078 Dijon, France and UMR 5584 CNRS, {\sf johannes.sjostrand@u-bourgogne.fr}}}
\date{}

\begin{document}
\maketitle

\vspace*{1cm}
\noindent
{\bf Abstract}: We study Toeplitz operators on the Bargmann space, with Toeplitz symbols that are exponentials of complex quadratic forms, from the point of view of Fourier integral operators in the complex domain. Sufficient conditions are established for the composition of two such operators to be a Toeplitz operator.

\tableofcontents
\section{Introduction and statement of results}
\setcounter{equation}{0}
\label{sec_introduction}
Various algebras of pseudodifferential and Fourier integral operators have long played a fundamental role in PDE and applications~\cite{H_book},~\cite{Sj94},~\cite{Zw12},~\cite{CoGrNiRo}. Parallel developments in the setting of Toeplitz operators on exponentially weighted spaces of holomorphic functions on $\comp^n$ (Bargmann spaces), motivated in part by the general ideas of quantization~\cite{Be74},~\cite{LC94}, are comparatively more recent, and may present some surprises~\cite{BC94},~\cite{LC01}. In the recent series of papers~\cite{CoHiSj},~\cite{CoHiSjWh},~\cite{CoHiSj23}, certain links have been established between the theory of Toeplitz operators on the Bargmann space and Fourier integral operators (FIOs) in the complex domain. In particular, the point of view of complex FIOs has been used in those works to establish the validity of the Berger-Coburn conjecture~\cite{BC94}, stating that a Toeplitz operator is bounded precisely when its Weyl symbol is bounded, for Toeplitz symbols that are exponentials of complex quadratic forms.

\medskip
\noindent
A basic, still unresolved issue concerns the problem of composing Toeplitz operators~\cite{LC19}, and the work~\cite{LC01} gives an explicit example of a bounded Toeplitz operator on the Bargmann space, whose square cannot be approximated by bounded Toeplitz operators in the operator norm. Now the Toeplitz symbol of the operator in the example of~\cite{LC01} is given by an exponential of a complex quadratic form, which makes it natural therefore to apply the machinery developed in~\cite{CoHiSj},~\cite{CoHiSjWh},~\cite{CoHiSj23}, to a systematic study of the composition of such metaplectic Toeplitz operators acting on the Bargmann space. The purpose of this paper is precisely to carry out such a study and to derive sufficient conditions for the composition of two metaplectic Toeplitz operators to be an operator of the same class. In addition to placing the example of~\cite{LC01} into a more conceptual framework, it seems that by doing so, one gains some additional insight into the analytic structure of the space of metaplectic Toeplitz operators, while also clarifying the link to the Weyl quantization. Let us now proceed to describe the assumptions and state the main results of this work.

\bigskip
\noindent
Let $\Phi_0$ be a strictly plurisubharmonic quadratic form on $\comp^n$ and let us introduce the Bargmann space
\begeq
\label{eq1.1}
H_{\Phi_0}(\comp^n) = L^2(\comp^n, e^{-2\Phi_0} L(dx)) \cap {\rm Hol}(\comp^n),
\endeq
with $L(dx)$ being the Lebesgue measure on $\comp^n$. We have the orthogonal projection
\begeq
\label{eq1.2}
\Pi_{\Phi_0}: L^2(\comp^n,e^{-2\Phi_0} L(dx)) \rightarrow H_{\Phi_0}(\comp^n).
\endeq

\medskip
\noindent
Let $q$ be a complex valued quadratic form on $\comp^n$ and assume that
\begeq
\label{eq1.3}
{\rm Re}\, q(x) < \Phi_{\rm herm}(x) := (1/2)\left(\Phi_0(x) + \Phi_0(ix)\right),\quad x \neq 0.
\endeq
In this work, we shall be concerned with (bounded) Toeplitz operators of the form
\begeq
\label{eq1.4}
{\rm Top}(e^q) = \Pi_{\Phi_0} \circ e^{q} \circ \Pi_{\Phi_0}: H_{\Phi_0}(\comp^n) \rightarrow H_{\Phi_0}(\comp^n).
\endeq
Any such operator can be represented as the Weyl quantization,
\begeq
\label{eq1.4.1}
{\rm Top}(e^q) = a^w(x,D_x),
\endeq
see~\cite{Sj95},~\cite{Zw12}, where the Weyl symbol $a$ is given by
\begeq
\label{eq1.4.2}
a\left(x,\xi\right)  = \left(\exp\left(\frac{1}{4} \left(\Phi''_{0,x\overline{x}}\right)^{-1} \partial_x \cdot \partial_{\overline{x}}\right)e^q\right)(x), \quad (x,\xi) \in \Lambda_{\Phi_0}.
\endeq
Here we have introduced the real linear subspace
\begeq
\label{eq1.5}
\Lambda_{\Phi_0} = \left\{\left(x,\frac{2}{i}\frac{\partial \Phi_0}{\partial x}(x)\right), \, x\in \comp^n\right\} \subset \comp^n_x \times \comp^n_{\xi} = \comp^{2n},
\endeq
which is I-Lagrangian and R-symplectic, in the sense that the restriction of the complex symplectic form on $\comp^{2n}$ to $\Lambda_{\Phi_0}$ is real and non-degenerate. In particular, $\Lambda_{\Phi_0}$ is maximally totally real, so that its complexification is given by $\comp^{2n}$.

\medskip
\noindent
As observed in~\cite{CoHiSj}, an application of the method of quadratic stationary phase to (\ref{eq1.4.2}) allows us to write,
\begeq
\label{eq1.7}
a(x,\xi) = C\, \exp(i F(x,\xi)), \quad (x,\xi) \in \Lambda_{\Phi_0},
\endeq
for some constant $C\neq 0$, where $F$ is a holomorphic quadratic form on $\comp^{2n}$. We let
\begeq
\label{eq1.8}
{\cal F}= \frac{1}{2} \begin{pmatrix}F''_{\xi x} &F''_{\xi \xi }\\
-F''_{xx} &-F''_{x\xi }\end{pmatrix}
\endeq
be the fundamental matrix of $F$ and assume, following~\cite{CoHiSj}, that $\pm 1 \notin {\rm Spec}({\cal F})$.

\bigskip
\noindent
Let $\widetilde{q}$ be a second complex valued quadratic form on $\comp^n$ satisfying, similarly to (\ref{eq1.3}),
\begeq
\label{eq1.9}
{\rm Re}\, \widetilde{q}(x) < \Phi_{\rm herm}(x),\quad x \neq 0.
\endeq
Letting $\widetilde{a}\in C^{\infty}(\Lambda_{\Phi_0})$ be the Weyl symbol of the Toeplitz operator ${\rm Top}(e^{\widetilde{q}})$, let us write as in (\ref{eq1.7}),
\begeq
\label{eq1.10}
\widetilde{a}(x,\xi) = \widetilde{C}\, \exp(i\widetilde{F}(x,\xi)), \quad (x,\xi) \in \Lambda_{\Phi_0},
\endeq
for a holomorphic quadratic form $\widetilde{F}$ on $\comp^{2n}$ and a constant $\widetilde{C} \neq 0$. Let $\widetilde{{\cal F}}$ be the fundamental matrix of $\widetilde{F}$ and assume that $\pm 1 \notin {\rm Spec}(\widetilde{{\cal F}})$.

\bigskip
\noindent
We shall assume that the Weyl symbols of ${\rm Top}(e^q)$, ${\rm Top}(e^{\widetilde{q}})$ satisfy
\begeq
\label{eq1.11}
a \in L^{\infty}(\Lambda_{\Phi_0}), \quad \widetilde{a} \in L^{\infty}(\Lambda_{\Phi_0}).
\endeq
As established in~\cite{CoHiSj},~\cite{CoHiSj23}, each of the assumptions in (\ref{eq1.11}) is equivalent to the boundedness of the corresponding operator,
\begeq
\label{eq1.12}
{\rm Top}(e^q): H_{\Phi_0}(\comp^n) \rightarrow H_{\Phi_0}(\comp^n),\quad {\rm Top}(e^{\widetilde{q}}): H_{\Phi_0}(\comp^n) \rightarrow H_{\Phi_0}(\comp^n).
\endeq
The following is the first main result of this work, with ${\rm Op}^w$ denoting the Weyl quantization in the complex domain, see~\cite{Sj95},~\cite[Chapter 13]{Zw12},~\cite{HiSj15}.

\begin{theo}
\label{theo_main_1}
Let $\Phi_0$ be a strictly plurisubharmonic quadratic form on $\comp^n$, and let $q$, $\widetilde{q}$ be complex valued quadratic forms on $\comp^n$ such that {\rm (\ref{eq1.3})}, {\rm (\ref{eq1.9})} hold. Let $a$, $\widetilde{a}\in C^{\infty}(\Lambda_{\Phi_0})$ be the Weyl symbols of the Toeplitz operators ${\rm Top}(e^q)$, ${\rm Top}(e^{\widetilde{q}})$, respectively, which are of the form {\rm (\ref{eq1.7})}, {\rm (\ref{eq1.10})}, and let us assume that $\pm 1 \notin {\rm Spec}({\cal F})$, $\pm 1 \notin {\rm Spec}(\widetilde{{\cal F}})$. Here ${\cal F}$, $\widetilde{{\cal F}}$ are the fundamental matrices of the quadratic forms $F$, $\widetilde{F}$, respectively. Assume that {\rm (\ref{eq1.11})} holds and that $-1\notin {\rm Spec}(\widetilde{{\cal F}}{\cal F})$. We have
\begeq
\label{eq1.13}
{\rm Top}(e^{\widetilde{q}})\circ {\rm Top}(e^q) = C\, {\rm Op}^w(e^{i\widehat{F}}): H_{\Phi_0}(\comp^n) \rightarrow H_{\Phi_0}(\comp^n),
\endeq
for some $0\neq C\in \comp$. Here $\widehat{F}$ is a holomorphic quadratic form on $\comp^{2n}$ such that
\begeq
\label{eq1.13.1}
{\rm Im}\, \widehat{F}|_{\Lambda_{\Phi_0}}\geq 0.
\endeq
The fundamental matrix $\widehat{{\cal F}}$ of $\widehat{F}$ given by
\begeq
\label{eq1.14}
\widehat{{\cal F}} = (1 + {\cal F})(1 + \widetilde{{\cal F}}{\cal F})^{-1} (1 + \widetilde{{\cal F}}) -1.
\endeq
\end{theo}

\bigskip
\noindent
Before stating the second main result, we shall recall the notion of polarization of a quadratic form on $\comp^n$. Given a complex valued quadratic form $f(x)$ on $\comp^n$, the polarization $f^{\pi}(x,y)$ of $f$ is the unique holomorphic quadratic form on $\comp^{2n}_{x,y}$ such that $f^{\pi}(x,\overline{x}) = f(x)$, $x\in \comp^n$. We shall use this system of notation below, with the following two exceptions: the polarization of the positive definite Hermitian form $\Phi_{{\rm herm}}$ in (\ref{eq1.3}) will be denoted by $\Psi_{{\rm herm}}$ and the polarization of the strictly plurisubharmonic quadratic form $\Phi_0$ will be denoted by $\Psi_0$.

\medskip
\noindent
Let $G$ be a holomorphic quadratic form on $\comp^{2n}$ and let us observe that the polarization of the quadratic form
\begeq
\label{eq1.15.01}
f(x) = G\left(x,\frac{2}{i}\frac{\partial \Phi_0}{\partial x}(x)\right),\quad x\in \comp^n,
\endeq
is given by
\begeq
\label{eq1.15.02}
f^{\pi}(x,y) = G\left(x,\frac{2}{i}\frac{\partial \Psi_0}{\partial x}(x,y)\right),\quad x,y \in \comp^n.
\endeq

\medskip
\noindent
\begin{theo}
\label{theo_main_2}
Let $\Phi_0$ be a strictly plurisubharmonic quadratic form on $\comp^n$, and let $G$ be a holomorphic quadratic form on $\comp^{2n}$ such that
\begeq
\label{eq1.15.02.1}
{\rm Im}\,G|_{\Lambda_{\Phi_0}}\geq 0.
\endeq
Assume that the fundamental matrix of $G$ does not have the eigenvalues $\pm 1$. Assume also that the holomorphic quadratic form
\begeq
\label{eq1.15.03}
iG\left(x,\frac{2}{i}\frac{\partial \Psi_0}{\partial x}(x,z)\right) + 4\Psi_{{\rm herm}}(x,z),\quad (x,z)\in \comp^{2n}
\endeq
is non-degenerate, and let us set
\begeq
\label{eq1.15.04}
Q^{\pi}(y,\theta) = {\rm vc}_{x,z}\left(4\Psi_{{\rm herm}}(x-y,z-\theta) + iG\left(x,\frac{2}{i}\frac{\partial \Psi_0}{\partial x}(x,z)\right)\right), \quad (y,\theta) \in \comp^{2n},
\endeq
where ${\rm vc}_{x,z}$ indicates that we take the critical value with respect to $x, z$. Here the critical value is attained at a unique critical point which is non-degenerate. Assume furthermore that the restriction $Q(y) = Q^{\pi}(y,\overline{y})$ of the holomorphic quadratic form $Q^{\pi}$ on $\comp^{2n}$ to the anti-diagonal satisfies
\begeq
\label{eq1.15.1}
{\rm Re}\, Q(y) < \Phi_{{\rm herm}}(y),\quad 0\neq y \in \comp^n.
\endeq
Then the Weyl quantization ${\rm Op}^w(e^{iG})$ is a Toeplitz operator, bounded on $H_{\Phi_0}(\comp^n)$, with the Toeplitz symbol of the form $C e^{Q}$, for some constant $C \neq 0$.
\end{theo}

\medskip
\noindent
{\it Remark}. Applying Theorem \ref{theo_main_2} in the case when $G = \widehat{F}$ in Theorem \ref{theo_main_1}, we obtain a general criterion for when the composition of two metaplectic Toeplitz operators in (\ref{eq1.13}) is a Toeplitz operator.

\medskip
\noindent
{\it Remark}. As we shall recall in Section \ref{sec_Examples} below, if the condition (\ref{eq1.15.1}) fails, then the bounded operator ${\rm Op}^w(e^{iG})$ in Theorem \ref{theo_main_2} does not need to be a Toeplitz operator.

\bigskip
\noindent
The composition problem for Toeplitz operators on the Bargmann space $H_{\Phi_0}(\comp^n)$ has been studied in~\cite{Ba},~\cite{LC01}, see also~\cite{V1},~\cite{HuVi}, as well as~\cite{Sj95},~\cite[Chapter 13]{Zw12} for the semiclassical case. In fact, the situation is particularly satisfying in the latter case, where the composition calculus often works with ${\cal O}(h^{\infty})$ -- errors. To recall a rough statement of it, following~\cite{Sj95},~\cite[Theorem 13.11]{Zw12},~\cite{LC19},~\cite{GevreyI}, let $p_1,p_2 \in C^{\infty}(\comp^n)$ be such that $\partial_x^{\alpha} \partial_{\overline{x}}^{\beta}p_j \in L^{\infty}(\comp^n)$, $j = 1,2$, for all $\alpha$, $\beta \in \nat^n$, and let us consider the semiclassical Toeplitz quantizations
\begeq
\label{eq1.16}
{\rm Top}_h(p_j) = \Pi_{\Phi_0,h} \circ p_j \circ \Pi_{\Phi_0,h} = {\cal O}(1): H_{\Phi_0,h}(\comp^n) \rightarrow H_{\Phi_0,h}(\comp^n).
\endeq
Here, similarly to (\ref{eq1.1}), we set $H_{\Phi_0,h}(\comp^n) = L^2(\comp^n, e^{-2\Phi_0/h}L(dx))\cap {\rm Hol}(\comp^n)$, and
$$
\Pi_{\Phi_0,h}: L^2(\comp^n, e^{-2\Phi_0/h}L(dx)) \rightarrow H_{\Phi_0,h}(\comp^n)
$$
is the orthogonal projection. We then have
\begeq
\label{eq1.17}
{\rm Top}_h(p_1) {\rm Top}_h(p_2) - {\rm Top}_h(p) = {\cal O}(h^{\infty}): H_{\Phi_0,h}(\comp^n) \rightarrow H_{\Phi_0,h}(\comp^n).
\endeq
Here $p\in C^{\infty}(\comp^n)$ admits a complete asymptotic expansion in integer powers of $h$, as $h\rightarrow 0^+$, that we shall only recall in the case when $\displaystyle \Phi_0(x) = \frac{\abs{x}^2}{4}$, see~\cite[Theorem 13.11]{Zw12} for the case of a general quadratic weight,
\begeq
\label{eq1.18}
p(x) \sim \sum_{\abs{\alpha}\geq 0} \frac{(-2h)^{\abs{\alpha}}}{\alpha!}\partial_x^{\alpha} p_1(x) \partial_{\overline{x}}^{\alpha} p_2(x),\quad x \in \comp^n.
\endeq
Accurate remainder bounds in the semiclassical expansion (\ref{eq1.18}) have been established in~\cite{ChPo}. In the non-semiclassical case, i.e. for $h=1$, the composition formula (\ref{eq1.18}) is still valid and becomes exact when the Toeplitz symbols $p_1$, $p_2$ are polynomials in $x$, $\overline{x}$, see~\cite{LC01}. Furthermore, it also holds, with absolute convergence, for $p_1$, $p_2$ which are Fourier-Stieltjes transforms of compactly supported measures on $\comp^n$,~\cite{LC01}. The Toeplitz symbols considered in this work may be unbounded, exhibiting some super-exponential growth at infinity, and when understanding the composition of the corresponding operators we shall rely crucially on the complex FIO point of view, developed in~\cite{CoHiSj},~\cite{CoHiSjWh},~\cite{CoHiSj23}.

\medskip
\noindent
{\it Remark}. As we also observed in~\cite{CoHiSj23}, while very special, the Toeplitz symbols considered here, given by exponentials of complex quadratic forms, may still be of some interest since the class of the associated Toeplitz operators includes those that are "at the edge" of boundedness, with the unboundedness of the symbols attenuated by their fast oscillations at infinity. As such, it has also been exploited as a source of various counter-examples, see~\cite{BCH},~\cite{LC01}.

\medskip
\noindent
The plan of the paper is as follows. In Section \ref{sec_complex_FIO}, we review some essentially well known results concerning the composition of metaplectic Fourier integral operators in the complex domain associated to complex linear canonical transformations, that are positive relative to the maximally totally real subspace $\Lambda_{\Phi_0}$ in (\ref{eq1.5}). This discussion is specialized in Section \ref{sec_Weyl} to FIOs given as Weyl quantizations of symbols of the form $e^{iF(x,\xi)}$, where $F$ is a holomorphic quadratic form on $\comp^{2n}$, and we show that under mild additional assumptions, the composition of two such operators is again an operator of this form. The proofs of Theorem \ref{theo_main_1} and \ref{theo_main_2} are then completed in Section \ref{sec_Toeplitz}, by passing from the Toeplitz symbols to the Weyl ones, along the lines of~\cite{CoHiSj},~\cite{CoHiSjWh}, and then back, by means of a well known critical value inversion formula, somewhat in the spirit of the inversion formula for the Legendre transformation. In Section \ref{sec_Examples} we discuss composition properties of some explicit families of metaplectic Toeplitz operators on a model Bargmann space, closely related to the example of~\cite{LC01}, and in particular we illustrate Theorems \ref{theo_main_1} and \ref{theo_main_2} in this case. Appendix A, finally, is devoted to some remarks concerning adjoints of complex metaplectic FIOs quantizing positive complex linear canonical transformations, realized as linear continuous maps between spaces of entire holomorphic functions with quadratic exponential weights. We compute the canonical transformation associated to the complex adjoint of such an operator.

\section{Composition of metaplectic FIOs in the complex domain}
\label{sec_complex_FIO}
\setcounter{equation}{0}
The discussion in this section is essentially well known, see~\cite[Chapter 4]{Sj82},~\cite{CGHS},~\cite{CoHiSj}, and serves as a convenient starting point for us. Let
\begeq
\label{eq2.1}
\kappa: \comp^{2n} \rightarrow \comp^{2n}
\endeq
be a complex linear canonical transformation, and let $\varphi(x,y,\theta)$ be a holomorphic quadratic form on $\comp^n_x \times \comp^n_y \times \comp^N_{\theta}$, which is a non-degenerate phase function in the sense of H\"ormander~\cite{H_FIOI},
\begeq
\label{eq2.2}
{\rm rank}\, \left(\varphi''_{\theta x}\,\,\, \varphi''_{\theta y}\,\,\, \varphi''_{\theta \theta}\right) = N,
\endeq
generating the graph of $\kappa$, so that
\begeq
\label{eq2.3}
\kappa: \comp^{2n}\ni (y,-\varphi'_y(x,y,\theta)) \mapsto (x,\varphi'_x(x,y,\theta))\in \comp^{2n},\quad \varphi'_{\theta}(x,y,\theta) = 0.
\endeq
For future reference, let us recall from~\cite{CGHS} that the fact that the canonical relation
\begeq
\label{eq2.3.0.1}
\left\{(x,\varphi'_x(x,y,\theta); y, -\varphi'_y(x,y,\theta)),\, \varphi'_{\theta}(x,y,\theta) = 0\right\} \subset \comp^{2n}\times \comp^{2n}
\endeq
is the graph of a linear canonical transformation is equivalent to the following condition,
\begeq
\label{eq2.3.0.2}
{\rm det}\, \begin{pmatrix}
\varphi''_{xy} & \varphi''_{x\theta} \\
\varphi''_{\theta y} & \varphi''_{\theta \theta} \\
\end{pmatrix} \neq 0.
\endeq

\bigskip
\noindent
Assume that
\begeq
\label{eq2.3.1}
\kappa(\Lambda_{\Phi_0}) = \Lambda_{\Phi_1},
\endeq
where $\Phi_0$, $\Phi_1$ are strictly plurisubharmonic quadratic forms on $\comp^n$. Here we have set as in (\ref{eq1.5}),
\begeq
\label{eq2.4}
\Lambda_{\Phi_j} = \left\{\left(x,\frac{2}{i}\frac{\partial \Phi_j}{\partial x}(x)\right), \, x\in \comp^n\right\} \subset \comp^{2n} = \comp^n_x \times \comp^n_{\xi},
\endeq
for $j= 0,1$. It follows from~\cite{Sj82},~\cite[Appendix B]{CGHS} that the plurisubharmonic quadratic form
\begeq
\label{eq2.5}
\comp^n \times \comp^N \ni (y,\theta) \mapsto -{\rm Im}\, \varphi(0,y,\theta) + \Phi_0(y)
\endeq
is non-degenerate of signature $(n+N,n+N)$. Let $(y_c(x),\theta_c(x))\in \comp^n \times \comp^N$ be the unique critical point of
\begeq
\label{eq2.6}
\comp^n \times \comp^N \ni (y,\theta) \mapsto -{\rm Im}\, \varphi(x,y,\theta) + \Phi_0(y),
\endeq
for each $x\in \comp^n$, and let us recall from~\cite[Appendix B]{CGHS} that
\begeq
\label{eq2.7}
\Phi_1(x) = {\rm vc}_{y,\theta} \left(-{\rm Im}\, \varphi(x,y,\theta) + \Phi_0(y)\right).
\endeq

\medskip
\noindent
It follows that there exists an affine subspace $\Gamma(x) \subset \comp^{n+N}_{y,\theta}$ of real dimension $n+N$, passing through the critical point $(y_c(x),\theta_c(x))$ such that
$$
-{\rm Im}\, \varphi(x,y,\theta) + \Phi_0(y) \leq \Phi(x) - \frac{1}{C} {\rm dist}\left((y,\theta), \left(y_c(x),\theta_c(x)\right)\right)^2,
$$
along $\Gamma(x)$. In such a situation, here and below, we say that $\Gamma(x) \subset \comp^{n+N}_{y,\theta}$ is a good contour for the plurisubharmonic function
$$
\comp^n \times \comp^N \ni (y,\theta) \mapsto -{\rm Im}\,\varphi(x,y,\theta) + \Phi_0(y).
$$
We conclude, following~\cite{Sj82},~\cite[Appendix B]{CGHS} that the corresponding realization of a Fourier integral operator $A$ quantizing $\kappa$,
\begeq
\label{eq2.8}
A_{\Gamma} u(x) = \int\!\!\!\int_{\Gamma(x)} e^{i\varphi(x,y,\theta)} a u(y)\, dy\, d\theta, \quad a \in \comp,
\endeq
defines a bounded linear map,
\begeq
\label{eq2.9}
A_{\Gamma} = A: H_{\Phi_0}(\comp^n) \rightarrow H_{\Phi_1}(\comp^n).
\endeq
Here the Bargmann space $H_{\Phi_0}(\comp^n)$ is defined in (\ref{eq1.1}), with the space $H_{\Phi_1}(\comp^n)$ having an analogous definition.

\medskip
\noindent
Let next $\widetilde{\kappa}: \comp^{2n} \rightarrow \comp^{2n}$ be a second complex linear canonical transformation, and let $\psi(x,y,w)$ be a
holomorphic quadratic form on $\comp^n_x \times \comp^n_y \times \comp^M_{w}$, which is a non-degenerate phase function in the sense of H\"ormander, such that
\begeq
\label{eq2.10}
\widetilde{\kappa}: \comp^{2n}\ni (y,-\psi'_y(x,y,w)) \mapsto (x,\psi'_x(x,y,w))\in \comp^{2n},\quad \psi'_{w}(x,y,w) = 0.
\endeq
Similarly to (\ref{eq2.3.1}), assume that
\begeq
\label{eq2.11}
\widetilde{\kappa}(\Lambda_{\Phi_1}) = \Lambda_{\Phi_2},
\endeq
where $\Phi_2$ is a strictly plurisubharmonic quadratic form on $\comp^n$. Letting $\widetilde{\Gamma}(x) \subset \comp^{n+M}_{y,w}$ be a good contour for the plurisubharmonic function
$$
\comp^n \times \comp^M \ni (y,w) \mapsto -{\rm Im}\,\psi(x,y,w) + \Phi_1(y),
$$
we consider the corresponding realization of a Fourier integral operator $B$ quantizing $\widetilde{\kappa}$,
\begeq
\label{eq2.12}
B_{\widetilde{\Gamma}} u(x) = \int\!\!\!\int_{\widetilde{\Gamma}(x)} e^{i\psi(x,y,w)} b\, u(y)\, dy\, dw, \quad b \in \comp,
\endeq
defining a bounded linear map,
\begeq
\label{eq2.13}
B_{\widetilde{\Gamma}} = B: H_{\Phi_1}(\comp^n) \rightarrow H_{\Phi_2}(\comp^n).
\endeq

\bigskip
\noindent
The composition $B_{\widetilde{\Gamma}}\circ A_{\Gamma}$ takes the form
\begeq
\label{eq2.14}
\left(B_{\widetilde{\Gamma}}\circ A_{\Gamma} u\right)(x) = \int\!\!\!\int\!\!\!\int\!\!\!\int_{\widehat{\Gamma}(x)} e^{i\left(\psi(x,z,w) + \varphi(z,y,\theta)\right)} ab\, u(y)\, dy\,d\theta\,dz\,dw,
\endeq
where $\widehat{\Gamma}(x) \subset \comp^n_z \times \comp^M_w \times \comp^n_y \times \comp^N_{\theta}$ is the composed contour of real dimension $2n+N+M$ given by
\begeq
\label{eq2.15}
\widehat{\Gamma}(x) = \left\{(z,w,y,\theta); (z,w)\in \widetilde{\Gamma}(x),\,\, (y,\theta)\in \Gamma(z)\right\}.
\endeq
Let us set
\begeq
\label{eq2.16}
\Phi(x,y;z,w,\theta) = \psi(x,z,w) + \varphi(z,y,\theta),
\endeq
with $\chi = (z,w,\theta)\in \comp^n_z \times \comp^M_w \times \comp^N_{\theta}$ viewed as the fiber variables. We claim that the holomorphic quadratic form $\Phi(x,y;\chi)$ is a non-degenerate phase function in the sense of H\"ormander, and when verifying the claim we proceed similarly to~\cite[Chapter 6]{Sogge}. We need to show that the $(n+M+N)\times (n+n+n+M+N)$ matrix
\begeq
\label{eq2.19.3}
\left(\Phi''_{\chi x}\,\,\, \Phi''_{\chi y}\,\,\, \Phi''_{\chi \chi}\right) = \begin{pmatrix}
\Phi''_{zx} & \Phi''_{zy} & \Phi''_{zz} & \Phi''_{zw} & \Phi''_{z\theta} \\
\Phi''_{wx} & \Phi''_{wy} & \Phi''_{wz} & \Phi''_{ww} & \Phi''_{w\theta} \\
\Phi''_{\theta x} & \Phi''_{\theta y} & \Phi''_{\theta z} & \Phi''_{\theta w} & \Phi''_{\theta \theta} \\
\end{pmatrix}
\endeq
has full rank, and using (\ref{eq2.16}) we see that the matrix in (\ref{eq2.19.3}) is of the form
\begeq
\label{eq2.19.4}
\begin{pmatrix}
\psi''_{zx} & \varphi''_{zy} & \psi''_{zz} + \varphi''_{zz} & \psi''_{zw} & \varphi''_{z\theta} \\
\psi''_{wx} & 0_{M\times n} & \psi''_{wz} & \psi''_{ww} & 0_{M\times N} \\
0_{N\times n} & \varphi''_{\theta y} & \varphi''_{\theta z} & 0_{N\times M} & \varphi''_{\theta \theta} \\
\end{pmatrix}.
\endeq
Here in view (\ref{eq2.3.0.2}), the $(n+N)\times (n+N)$ matrix
\begeq
\label{eq2.19.5}
\begin{pmatrix}
\varphi''_{zy} & \varphi''_{z\theta} \\
\varphi''_{\theta y} & \varphi''_{\theta \theta} \\
\end{pmatrix}
\endeq
is non-degenerate, and we have
\begeq
\label{eq2.19.6}
{\rm rank}\, \left(\psi''_{w x}\,\,\, \psi''_{w z}\,\,\, \psi''_{w w}\right) = M.
\endeq
Let $B$ be an invertible $M\times M$ matrix, whose columns are among the columns of the matrix $\left(\psi''_{w x}\,\,\, \psi''_{w z}\,\,\, \psi''_{w w}\right)$. Observing that an $(n+M+N)\times (n+M+N)$ matrix of the form
\begeq
\label{eq2.19.8}
\begin{pmatrix}
A & \varphi''_{zy} & \varphi''_{z\theta} \\
B & 0_{M\times n} & 0_{M\times N} \\
C & \varphi''_{\theta y} & \varphi''_{\theta \theta} \\
\end{pmatrix}
\endeq
is non-degenerate, independently of matrices $A$, $C$ of size $n\times M$ and $N\times M$, respectively, we conclude that the matrix in (\ref{eq2.19.3}) is of full rank, giving the claim.

\medskip
\noindent
It follows that the associated canonical relation
\begeq
\label{eq2.17}
\comp^{2n} \ni (y,-\Phi'_y(x,y;\chi)) \mapsto (x,\Phi'_x(x,y;\chi))\in \comp^{2n},\quad \Phi'_z = 0,\,\,\Phi'_w = 0,\,\,\Phi'_{\theta} = 0,
\endeq
is of dimension $2n$ and is given by
\begeq
\label{eq2.18}
\comp^{2n}\ni (y,-\varphi'_y(z,y,\theta))\mapsto (x,\psi'_x(x,z,w))\in \comp^{2n},
\endeq
\begeq
\label{eq2.19}
\psi'_z(x,z,w) + \varphi'_z(z,y,\theta)=0,\,\,\psi'_w(x,z,w) = 0,\,\, \varphi'_{\theta}(z,y,\theta) = 0.
\endeq
It is therefore clear that the canonical relation (\ref{eq2.17}) is the graph of the canonical transformation $\widetilde{\kappa}\circ \kappa$.

\bigskip
\noindent
We see furthermore, directly from the definitions, that the plurisubharmonic quadratic form
\begeq
\label{eq2.20}
\comp^n_z \times \comp^M_w \times \comp^n_y \times \comp^N_{\theta}\ni (z,w,y,\theta)\mapsto -{\rm Im}\,\psi(0,z,w) - {\rm Im}\, \varphi(z,y,\theta) + \Phi_0(y)
\endeq
is negative definite along the contour $\widehat{\Gamma}(0)$ of real dimension $2n+N+M$, and therefore, the quadratic form (\ref{eq2.20}) is non-degenerate of signature $(2n+N+M,2n+N+M)$. It follows that the composed contour $\widehat{\Gamma}(x)$ is good for the function
\begeq
\label{eq2.21}
\comp^n_z \times \comp^M_w \times \comp^n_y \times \comp^N_{\theta}\ni (z,w,y,\theta)\mapsto -{\rm Im}\,\psi(x,z,w) - {\rm Im}\, \varphi(z,y,\theta) + \Phi_0(y),
\endeq
and we conclude therefore that the composition $B_{\widetilde{\Gamma}}\circ A_{\Gamma}$ is a realization of a Fourier integral operator quantizing the canonical transformation $\widetilde{\kappa}\circ \kappa$.

\bigskip
\noindent
Assume next that the complex linear canonical transformations $\kappa$, $\widetilde{\kappa}$ are positive relative to $\Lambda_{\Phi_0}$, so that
\begeq
\label{eq2.22}
\frac{1}{i} \biggl(\sigma(\kappa(\rho), \iota_{\Phi_0} \kappa(\rho)) - \sigma(\rho, \iota_{\Phi_0}(\rho))\biggr) \geq 0,\quad \rho \in \comp^{2n},
\endeq
and similarly for $\widetilde{\kappa}$, see~\cite{CoHiSj}. Here $\iota_{\Phi_0}: \comp^{2n} \rightarrow \comp^{2n}$ is the unique anti-linear involution such that $\iota|_{\Lambda_{\Phi_0}} = 1$, and
\begeq
\label{eq2.23}
\sigma = \sum_{j=1}^n d\xi_j \wedge dx_j
\endeq
is the complex symplectic form on $\comp^{2n} = \comp^n_x \times \comp^n_{\xi}$. It follows from~\cite[Theorem 1.1]{CoHiSj} that (\ref{eq2.3.1}) holds, with the strictly plurisubharmonic quadratic form $\Phi_1$ satisfying $\Phi_1 \leq \Phi_0$. Let us next check that (\ref{eq2.11}) holds as well, with $\Phi_2$ quadratic strictly plurisubharmonic on $\comp^n$ satisfying $\Phi_2 \leq \Phi_0$. When doing so, we observe that since $\widetilde{\kappa}$ is positive relative to $\Lambda_{\Phi_0}$, we have that the plurisubharmonic quadratic form
\begeq
\label{eq2.24}
\comp^n \times \comp^M \ni (y,w) \mapsto -{\rm Im}\, \psi(0,y,w) + \Phi_0(y)
\endeq
is non-degenerate of signature $(n+M,n+M)$. Using that
\begeq
\label{eq2.25}
-{\rm Im}\, \psi(0,y,w) + \Phi_1(y) \leq -{\rm Im}\, \psi(0,y,w) + \Phi_0(y),
\endeq
we conclude that since the left hand side in (\ref{eq2.25}) is a plurisubharmonic quadratic form, it is also non-degenerate of signature $(n+M,n+M)$. It follows that (\ref{eq2.11}) holds with
\begeq
\label{eq2.26}
\Phi_2(x) = {\rm vc}_{y,w}\left(-{\rm Im}\, \psi(x,y,w) + \Phi_1(y)\right).
\endeq
An application of the fundamental lemma of~\cite{Sj82} allows us to conclude that $\Phi_2$ is plurisubharmonic, and since the real linear subspace $\Lambda_{\Phi_2} = \widetilde{\kappa}(\Lambda_{\Phi_1})$ is R-symplectic, the plurisubharmonicity of $\Phi_2$ is necessarily strict. We also have
\begeq
\label{eq2.27}
\Phi_2(x) = {\rm vc}_{y,w}\left(-{\rm Im}\, \psi(x,y,w) + \Phi_1(y)\right) \leq {\rm vc}_{y,w}\left(-{\rm Im}\, \psi(x,y,w) + \Phi_0(y)\right),
\endeq
and the strictly plurisubharmonic quadratic form in the right hand side is $\leq \Phi_0$, in view of the positivity of $\widetilde{\kappa}$ relative to $\Lambda_{\Phi_0}$.

\bigskip
\noindent
We may summarize the discussion in this section in the following essentially well known result, see also~\cite[Proposition B.4]{CGHS}.

\begin{theo}
\label{theo1}
Let $\Phi_0$ be a strictly plurisubharmonic quadratic form on $\comp^n$, and let $\kappa$, $\widetilde{\kappa}: \comp^{2n} \rightarrow \comp^{2n}$ be complex linear canonical transformations that are positive relative to $\Lambda_{\Phi_0}$. Let $A$, $B$ be metaplectic Fourier integral operators quantizing $\kappa$, $\widetilde{\kappa}$, respectively, realized with the help of good contours. Then the composition $B \circ A$ is a Fourier integral operator associated to the canonical transformation $\widetilde{\kappa}\circ \kappa$, which is also positive relative to $\Lambda_{\Phi_0}$. The operator $B\circ A$ can be realized with the help of a good contour and we have that
$$
B\circ A: H_{\Phi_0}(\comp^n) \rightarrow H_{\Phi_0}(\comp^n)
$$
is bounded.
\end{theo}

\section{Composing complex Weyl quantizations}
\label{sec_Weyl}
\setcounter{equation}{0}
Let $\Phi_0$ be a strictly plurisubharmonic quadratic form on $\comp^n_x$ and let $F$ be a holomorphic quadratic form on $\comp^{2n}_{x,\xi}$ such that
\begeq
\label{eq3.1}
{\rm Im}\, F \geq 0 \quad \wrtext{along}\,\,\Lambda_{\Phi_0}.
\endeq
Here the real $2n$-dimensional linear subspace $\Lambda_{\Phi_0}\subset \comp^{2n}$ has been introduced in (\ref{eq1.5}). Assume that the fundamental matrix of $F$,
\begeq
\label{eq3.2}
{\cal F} = \frac{1}{2} \begin{pmatrix}F''_{\xi x} &F''_{\xi \xi }\\ -F''_{x x} &-F''_{x \xi}\end{pmatrix}
\endeq
satisfies
\begeq
\label{eq3.2.1}
\pm 1 \notin {\rm Spec}({\cal F}).
\endeq
Then, as explained in~\cite{CoHiSj},~\cite{CoHiSjWh}, the Weyl quantization ${\rm Op}^w(e^{iF})$ can be regarded as a Fourier integral operator in the complex domain associated to the complex linear canonical transformation
\begeq
\label{eq3.3}
\kappa = (1 - {\cal F}) \left(1 + {\cal F}\right)^{-1}.
\endeq
Here we may notice that the map
\begeq
\label{eq3.3.1}
\kappa +1 = (1 - {\cal F} + 1 + {\cal F})\left(1 + {\cal F}\right)^{-1} = 2 \left(1 + {\cal F}\right)^{-1}
\endeq
is bijective, and we have the inverse relation
\begeq
\label{eq3.3.2}
{\cal F} = (1+\kappa)^{-1}(1-\kappa).
\endeq
Furthermore, as we have seen in~\cite[Proposition B.1]{CoHiSj}, the assumption (\ref{eq3.1}) implies that $\kappa$ is positive relative to $\Lambda_{\Phi_0}$, i.e., that (\ref{eq2.22}) holds, and that the operator
\begeq
\label{eq3.3.3}
{\rm Op}^w(e^{iF}): H_{\Phi_0}(\comp^n) \rightarrow H_{\Phi_0}(\comp^n)
\endeq
is bounded.

\medskip
\noindent
Let $\widetilde{F}$ be a second holomorphic quadratic form on $\comp^{2n}_{x,\xi}$ such that
\begeq
\label{eq3.4}
{\rm Im}\, \widetilde{F} \geq 0 \quad \wrtext{along}\,\,\Lambda_{\Phi_0}.
\endeq
Assume also that $\pm 1 \notin {\rm Spec}(\widetilde{{\cal F}})$, where $\widetilde{{\cal F}}$ is the fundamental matrix of $\widetilde{F}$. It follows therefore from Theorem \ref{theo1} that the composition ${\rm Op}^w(e^{i\widetilde{F}})\circ {\rm Op}^w(e^{iF})$ is a Fourier integral operator associated to the complex linear canonical transformation
\begeq
\label{eq3.5}
\widehat{\kappa} := \widetilde{\kappa}\circ \kappa: \comp^{2n} \rightarrow \comp^{2n},
\endeq
which is positive relative to $\Lambda_{\Phi_0}$. Here $\widetilde{\kappa} = (1 - \widetilde{{\cal F}}) (1 + \widetilde{{\cal F}})^{-1}$. Assume that $-1\notin {\rm Spec}(\widehat{\kappa})$ and let us set, similarly to (\ref{eq3.3.2}),
\begeq
\label{eq3.6}
\widehat{{\cal F}} = (1+\widehat{\kappa})^{-1}(1-\widehat{\kappa}).
\endeq
Using the fact that $\widehat{\kappa}$ is canonical, we see that the complex linear map $\widehat{{\cal F}}$ is skew-symmetric with respect to $\sigma$,
\begeq
\label{eq3.6.1}
\widehat{{\cal F}} + \widehat{{\cal F}}^{\sigma} = 0,
\endeq
where $\widehat{{\cal F}}^{\sigma}$ is the symplectic transpose of $\widehat{{\cal F}}$, given by
$$
\sigma(\widehat{{\cal F}}\mu,\nu) = \sigma(\mu, \widehat{{\cal F}}^{\sigma}\nu),\quad \mu,\nu \in \comp^{2n}.
$$
Writing
\begeq
\label{eq3.7}
\sigma(\mu,\nu) = J\mu \cdot \nu, \quad \mu,\nu \in \comp^{2n},
\endeq
where
\begeq
\label{eq3.8}
J=\begin{pmatrix}0 &1\\ -1 &0\end{pmatrix},\quad J^t = -J,\quad J^2 = -1,
\endeq
we see that (\ref{eq3.6.1}) is equivalent to the statement that $J\widehat{{\cal F}}$ is symmetric. It follows that the holomorphic quadratic form
\begeq
\label{eq3.9}
\widehat{F}(\rho) = \sigma(\rho,\widehat{{\cal F}}\rho) = - J \widehat{{\cal F}}\rho\cdot \rho, \quad \rho \in \comp^{2n},
\endeq
satisfies
\begeq
\label{eq3.10}
\sigma(t,H_{\widehat{F}}(\rho)) = Jt\cdot H_{\widehat{F}}(\rho) = d \widehat{F}(\rho)\cdot t = -2J \widehat{{\cal F}}\rho\cdot t =
2\widehat{{\cal F}}\rho\cdot Jt,\quad t\in \comp^{2n}.
\endeq
We get $H_{\widehat{F}}(\rho) = 2 \widehat{{\cal F}}\rho$, where $H_{\widehat{F}}$ is the Hamilton vector field of $\widehat{F}$, and therefore, in view of (\ref{eq3.2}), we conclude that $\widehat{{\cal F}}$ is the fundamental matrix of the quadratic form $\widehat{F}$. Using (\ref{eq3.6}), we observe also that the linear map
\begeq
\label{eq3.11}
\widehat{{\cal F}} + 1 = (1+\widehat{\kappa})^{-1}(1-\widehat{\kappa} + 1 + \widehat{\kappa}) = 2 (1+\widehat{\kappa})^{-1}
\endeq
is bijective, and therefore $1 - \widehat{{\cal F}}$ is bijective as well. It follows furthermore  from (\ref{eq3.6}) that the canonical transformation $\widehat{\kappa}$ takes the form
\begeq
\label{eq3.12}
\left(1 + \widehat{{\cal F}}\right)\rho \mapsto \left(1 - \widehat{{\cal F}}\right)\rho,
\endeq
and recalling the positivity of $\widehat{\kappa}$ relative to $\Lambda_{\Phi_0}$, we conclude, following~\cite[Proposition B.1]{CoHiSj}, that the holomorphic quadratic form in (\ref{eq3.9}) satisfies
\begeq
\label{eq3.13}
{\rm Im}\, \widehat{F} \geq 0 \quad \wrtext{along}\,\,\Lambda_{\Phi_0}.
\endeq
It follows that the holomorphic quadratic form
$$
\comp^n_x \times \comp^n_y \times \comp^n_{\theta} \ni (x,y,\theta) \mapsto (x-y)\cdot \theta + \widehat{F}\left(\frac{x+y}{2},\theta\right)
$$
is a non-degenerate phase function in the sense of H\"ormander, which generates the positive complex linear canonical transformation
$\widehat{\kappa} = \widehat{\kappa}\circ \kappa$ in (\ref{eq3.5}).

\medskip
\noindent
The discussion above can be summarized in the following result.
\begin{prop}
\label{Weyl}
Let $\Phi_0$ be a strictly plurisubharmonic quadratic form on {\rm $\comp^n$}, and let $F$, $\widetilde{F}$ be holomorphic quadratic forms on {\rm $\comp^{2n}$} such that
$$
{\rm Im}\, F\left(x,\frac{2}{i}\frac{\partial \Phi_0}{\partial x}(x)\right)\geq 0, \quad {\rm Im}\, \widetilde{F}\left(x,\frac{2}{i}\frac{\partial \Phi_0}{\partial x}(x)\right)\geq 0,\quad x\in \comp^n.
$$
Assume that the fundamental matrices ${\cal F}$, $\widetilde{\cal F}$ of the quadratic forms $F$, $\widetilde{F}$, respectively, satisfy
$\pm 1 \notin {\rm Spec}({\cal F})$, $\pm 1 \notin {\rm Spec}(\widetilde{{\cal F}})$. Let
\begeq
\label{eq3.14}
\kappa = (1 - {\cal F}) \left(1 + {\cal F}\right)^{-1},\quad \widetilde{\kappa} = (1 - \widetilde{{\cal F}}) (1 + \widetilde{{\cal F}})^{-1},
\endeq
and assume that $-1\notin {\rm Spec}(\widetilde{\kappa}\circ \kappa)$. We have then
\begeq
\label{eq3.15}
{\rm Op}^w(e^{i\widetilde{F}})\circ {\rm Op}^w(e^{iF}) = C\, {\rm Op}^w(e^{i\widehat{F}}),
\endeq
for some constant $0 \neq C \in \comp$, where $\widehat{F}$ is a holomorphic quadratic form  on $\comp^{2n}$ satisfying
$$
{\rm Im}\, \widehat{F}\left(x,\frac{2}{i}\frac{\partial \Phi_0}{\partial x}(x)\right)\geq 0,\quad x\in \comp^n,
$$
and such that $\pm 1 \notin {\rm Spec}(\widehat{{\cal F}})$, where $\widehat{\cal F}$ is the fundamental matrix of $\widehat{F}$. The operators in {\rm (\ref{eq3.15})} are bounded: $H_{\Phi_0}(\comp^n) \rightarrow H_{\Phi_0}(\comp^n)$.
\end{prop}

\bigskip
\noindent
{\it Remark}. The constant $C\neq 0$ in (\ref{eq3.15}) can be computed, see~\cite{H_Mehler}.

\medskip
\noindent
{\it Remark}. Assume that one of the canonical transformations $\kappa$, $\widetilde{\kappa}$ in (\ref{eq3.14}) is strictly positive relative to $\Lambda_{\Phi_0}$, in the sense that the inequality in (\ref{eq2.22}) is strict, for all $\rho \neq 0$. It follows that $\widehat{\kappa} = \widetilde{\kappa}\circ \kappa$ is also strictly positive relative to $\Lambda_{\Phi_0}$,
\begeq
\label{eq3.16}
\frac{1}{i} \biggl(\sigma(\widehat{\kappa}(\rho), \iota_{\Phi_0} \widehat{\kappa}(\rho)) - \sigma(\rho, \iota_{\Phi_0}(\rho))\biggr) > 0,\quad 0\neq \rho \in \comp^{2n},
\endeq
and in particular $-1\notin {\rm Spec}(\widehat{\kappa})$. More generally, we may observe that the spectrum of a strictly positive complex linear canonical transformation avoids the set $\{\lambda\in \comp; \abs{\lambda} = 1\}$. We also recall from~\cite{CoHiSj} that the strict positivity of $\kappa$ in (\ref{eq3.14}) relative to $\Lambda_{\Phi_0}$ is equivalent to the ellipticity property,
\begeq
\label{eq3.17}
{\rm Im}\, F\left(x,\frac{2}{i}\frac{\partial \Phi_0}{\partial x}(x)\right) \asymp \abs{x}^2,\quad x\in \comp^n.
\endeq

\bigskip
\noindent
We shall finish this section by deriving an explicit expression for the fundamental matrix $\widehat{{\cal F}}$ of the quadratic form $\widehat{F}$ in (\ref{eq3.15}) in terms of the fundamental matrices ${\cal F}$, $\widetilde{{\cal F}}$. See also~\cite{Viola} for closely related computations.

\medskip
\noindent
Let us use (\ref{eq3.3.2}) and the corresponding expression for $\widetilde{{\cal F}}$ in terms of $\widetilde{\kappa}$ to write
\begeq
\label{eq3.18}
{\cal F} = (1+\kappa)^{-1}(1-\kappa), \quad \widetilde{{\cal F}} = (1+\widetilde{\kappa})^{-1}(1-\widetilde{\kappa}).
\endeq
It follows that
\begin{multline}
\label{eq3.19}
1 + \widetilde{{\cal F}}{\cal F} = 1 + (1+\widetilde{\kappa})^{-1}(1-\widetilde{\kappa})(1+\kappa)^{-1}(1-\kappa) \\
= (1+\widetilde{\kappa})^{-1}\biggl((1+\widetilde{\kappa})(1+\kappa) + (1-\widetilde{\kappa})(1-\kappa)\biggr)(1+\kappa)^{-1} \\
= 2 (1+\widetilde{\kappa})^{-1}\left(1 + \widetilde{\kappa}\kappa\right)(1+\kappa)^{-1} =
\frac{1}{2} (1 + \widetilde{{\cal F}})\left(1 + \widetilde{\kappa}\kappa\right)(1 + {\cal F}).
\end{multline}
Here in the last equality we have used (cf. (\ref{eq3.3.1})) that
\begeq
\label{eq3.20}
{\cal F} + 1 = 2 (1+\kappa)^{-1},\quad \widetilde{{\cal F}} + 1 = 2 (1+\widetilde{\kappa})^{-1}.
\endeq
Using (\ref{eq3.19}), we get that
\begeq
\label{eq3.21}
2 (1 + \widetilde{{\cal F}})^{-1} (1 + \widetilde{{\cal F}}{\cal F})(1 + {\cal F})^{-1} = 1 + \widetilde{\kappa}\kappa.
 \endeq
In Proposition \ref{Weyl} we have assumed that the linear map $1 + \widetilde{\kappa}\kappa = 1 + \widehat{\kappa}$ is bijective, and using (\ref{eq3.21}) we conclude that this assumption is equivalent to the bijectivity of $1 + \widetilde{{\cal F}}{\cal F}$. We get
\begeq
\label{eq3.22}
2 (1 + \widehat{\kappa})^{-1} = (1 + {\cal F})(1 + \widetilde{{\cal F}}{\cal F})^{-1} (1 + \widetilde{{\cal F}}),
\endeq
and recalling (\ref{eq3.11}) we conclude that
\begeq
\label{eq3.23}
\widehat{{\cal F}} = (1 + {\cal F})(1 + \widetilde{{\cal F}}{\cal F})^{-1} (1 + \widetilde{{\cal F}}) -1.
\endeq
Combining Proposition \ref{Weyl} with (\ref{eq3.23}), we shall be able to complete the proof of Theorem \ref{theo_main_1} in Section \ref{sec_Toeplitz} below, once we have recalled how to express a Toeplitz operator of the form (\ref{eq1.4}) as a Weyl quantization.

\section{From Toeplitz quantization to Weyl quantization and back}
\label{sec_Toeplitz}
\setcounter{equation}{0}
Let $\Phi_0$ be a strictly plurisubharmonic quadratic form on $\comp^n$ and let $q$ be a complex valued quadratic form on $\comp^n$ such that (\ref{eq1.3}) holds. From~\cite{CoHiSj},~\cite{CoHiSjWh} we recall that when equipped with the maximal domain
\begeq
\label{eq4.2}
{\cal D}({\rm Top}(e^q)) = \left\{u\in H_{\Phi_0}(\comp^n); e^{q} u \in L^2(\comp^n, e^{-2\Phi_0}L(dx))\right\},
\endeq
the Toeplitz operator
\begeq
\label{eq4.3}
{\rm Top}(e^{q}) = \Pi_{\Phi_0} \circ e^{q} \circ \Pi_{\Phi_0}: H_{\Phi_0}(\comp^n) \rightarrow H_{\Phi_0}(\comp^n)
\endeq
becomes densely defined. Here the orthogonal projection $\Pi_{\Phi_0}: L^2(\comp^n,e^{-2\Phi_0} L(dx)) \rightarrow H_{\Phi_0}(\comp^n)$ has been introduced in (\ref{eq1.2}).

\medskip
\noindent
{\it Remark}. Let $\Psi_0$ be the polarization of $\Phi_0$. Using the well known property
\begeq
\label{eq4.4.1}
2{\rm Re}\, \Psi_0(x,\overline{y}) - \Phi_0(x) - \Phi_0(y) = - \Phi_{\rm herm}(x-y) \asymp - \abs{x-y}^2,
\endeq
see~\cite{Sj95}, together with (\ref{eq1.3}), we obtain that
\begeq
\label{eq4.4.2}
e^{2\Psi_0(\cdot, \overline{y})} \in {\cal D}({\rm Top}(e^q)),\quad y\in \comp^n.
\endeq
We may also observe that the linear span of $\{e^{2\Psi_0(\cdot, \overline{y})}; y\in \comp^n\}$ is dense in $H_{\Phi_0}(\comp^n)$.

\bigskip
\noindent
Let us write, following~\cite{Sj95},~\cite{CoHiSj},
\begeq
\label{eq4.5}
{\rm Top}(e^q) = a^w(x,D_x),
\endeq
where $a\in C^{\infty}(\Lambda_{\Phi_0})$ is the Weyl symbol of the Toeplitz operator ${\rm Top}(e^{q})$, given by
\begeq
\label{eq4.6}
a\left(x,\xi\right)  = \left(\exp\left(\frac{1}{4} \left(\Phi''_{0,x\overline{x}}\right)^{-1} \partial_x \cdot \partial_{\overline{x}}\right)e^q\right)(x), \quad (x,\xi) \in \Lambda_{\Phi_0}.
\endeq
An application of the method of exact stationary phase allows us to conclude that
\begeq
\label{eq4.8}
a(x,\xi) = C\, \exp(i(F(x,\xi))), \quad (x,\xi) \in \Lambda_{\Phi_0},
\endeq
for some constant $C\neq 0$, where $F$ is a holomorphic quadratic form on $\comp^{2n}$. See also (\ref{eq4.11.5}) and the computations below. We may write therefore
\begeq
\label{eq4.9}
{\rm Top}(e^q) = C\,{\rm Op}^w(e^{iF}).
\endeq
In what follows, we shall assume that the fundamental matrix ${\cal F}$ of $F$ satisfies $\pm 1 \notin {\rm Spec}({\cal F})$.

\medskip
\noindent
Let $\widetilde{q}$ be a second complex valued quadratic form on $\comp^n$ satisfying (\ref{eq1.9}), and let us write similarly to (\ref{eq4.9}),
\begeq
\label{eq4.10}
{\rm Top}(e^{\widetilde{q}}) = \widetilde{C}\, {\rm Op}^w(e^{i\widetilde{F}}),\quad \widetilde{C} \neq 0.
\endeq
Here $\widetilde{F}$ is a holomorphic quadratic form on $\comp^{2n}$. Assume also that the fundamental matrix $\widetilde{{\cal F}}$ of $\widetilde{F}$ is such that $\pm 1 \notin {\rm Spec}(\widetilde{{\cal F}})$.

\bigskip
\noindent
Assume that the Weyl symbols satisfy
\begeq
\label{eq4.10.1}
e^{iF}\in L^{\infty}(\Lambda_{\Phi_0}), \quad e^{i\widetilde{F}}\in L^{\infty}(\Lambda_{\Phi_0}),
\endeq
and that $1 + \widetilde{{\cal F}}{\cal F}: \comp^{2n} \rightarrow \comp^{2n}$ is bijective. The discussion in Section \ref{sec_Weyl} applies therefore to the composition
${\rm Top}(e^{\widetilde{q}})\circ {\rm Top}(e^q)$, in view of (\ref{eq4.9}), (\ref{eq4.10}), as it stands, implying Theorem \ref{theo_main_1}.

\bigskip
\noindent
We shall now proceed to give a proof of Theorem \ref{theo_main_2}, and to this end, it will first be convenient to take a closer look at the formula for the Weyl symbol (\ref{eq4.6}), and to give a more explicit description of the Fourier multiplier in (\ref{eq4.6}). When doing so, we observe that the symbol of the second order constant coefficient differential operator on $\comp^n_x$,
\begeq
\label{eq4.11.1}
-\frac{1}{4} \left(\Phi''_{0,x\overline{x}}\right)^{-1} \partial_x \cdot \partial_{\overline{x}}
\endeq
is given by the positive definite quadratic form
\begeq
\label{eq4.11.2}
\frac{1}{2} \left(8\Phi''_{0,x\overline{x}}\right)^{-1} \overline{\xi}\cdot \xi,\quad \xi \in \comp^n.
\endeq
Here it will be convenient to recall from~\cite{DiSj} that the dual of a real valued non-degenerate quadratic form $\real^N \ni x \mapsto \displaystyle \frac{1}{2} Ax\cdot x$ is by definition the quadratic form $\real^N \ni \xi \mapsto \displaystyle \frac{1}{2} A^{-1}\xi\cdot \xi$. Assuming that the quadratic form is positive definite, we can express its dual as the Legendre transform,
\begeq
\label{eq4.11.2.1}
\frac{1}{2} A^{-1}\xi\cdot \xi = {\rm sup}_{x}\left(x\cdot \xi - \frac{1}{2} Ax\cdot x\right).
\endeq
It follows therefore that the dual of the quadratic form in (\ref{eq4.11.2}) is given by the positive definite quadratic form
\begeq
\label{eq4.11.3}
{\rm sup}_{\xi} \left({\rm Re}\, (\xi\cdot \overline{x}) -  \frac{1}{2} \left(8\Phi''_{0,x\overline{x}}\right)^{-1} \overline{\xi}\cdot \xi\right)
= 4\Phi''_{0,x\overline{x}}\overline{x}\cdot x = 4\Phi_{{\rm herm}}(x).
\endeq
Combining this observation with the standard formula,
\begeq
\label{eq4.11.4}
e^{-\frac{AD\cdot D}{2}}u(x) = \frac{1}{(2\pi)^{N/2}} \frac{1}{({\rm det} A)^{1/2}} \int e^{-\frac{A^{-1}y\cdot y}{2}}u(x-y)\,dy,
\endeq
where $A$ is an $N\times N$ real symmetric positive definite matrix and $u\in {\cal S}(\real^N)$, we conclude, in view of (\ref{eq4.6}), that we have
\begeq
\label{eq4.11.5}
a(x,\xi) = C_{\Phi_0} \int_{{\bf C}^n} \exp(-4 \Phi_{{\rm herm}}(x-y)) e^{q(y)}\, L(dy),\quad (x,\xi) \in \Lambda_{\Phi_0}.
\endeq
Here $C_{\Phi_0}\neq 0$ and the integral converges thanks to (\ref{eq1.3}).

\bigskip
\noindent
We shall next evaluate a general Gaussian integral of the form (\ref{eq4.11.5}). To this end, let $Q$ be a complex valued quadratic form on $\comp^n$ such that
\begeq
\label{eq4.11.6}
{\rm Re}\, Q(x) < \Phi_{{\rm herm}}(x),\quad 0\neq x \in \comp^n.
\endeq
Introducing the polarizations $Q^{\pi}$ of $Q$ and $\Psi_{{\rm herm}}$ of $\Phi_{{\rm herm}}$, we may write in view of (\ref{eq4.11.5}), for some $C\neq 0$,
\begin{multline}
\label{eq4.12}
\left(\exp\left(\frac{1}{4} \left(\Phi''_{0,x\overline{x}}\right)^{-1} \partial_x \cdot \partial_{\overline{x}}\right)e^Q\right)(x) \\
= C \int\!\!\!\int_{\Gamma} \exp\left(-4\Psi_{{\rm herm}}(x-y,\overline{x}-\theta) + Q^{\pi}(y,\theta)\right)\, dy\,d\theta.
\end{multline}
Here $\Gamma \subset \comp^{2n}_{y,\theta}$ is the contour given by $\theta = \overline{y}$ (the anti-diagonal). An application of~\cite[Proposition 2.1]{CoHiSj23} together with (\ref{eq4.11.6}) allows us to conclude that the holomorphic quadratic form
\begeq
\label{eq4.17.3}
\comp^{2n}_{y,\theta} \ni (y,\theta) \mapsto -4\Psi_{{\rm herm}}(y,\theta) + Q^{\pi}(y,\theta)
\endeq
is non-degenerate, and therefore the holomorphic function
\begeq
\label{eq4.17.4}
\comp^{2n}_{y,\theta} \ni (y,\theta) \mapsto -4\Psi_{{\rm herm}}(x-y,z-\theta) + Q^{\pi}(y,\theta)
\endeq
has a unique critical point which is non-degenerate, for each $(x,z)\in \comp^n\times \comp^n$. In view of the method of exact (quadratic) stationary phase and (\ref{eq4.12}), it is clear therefore that
\begin{multline}
\label{eq4.17.5}
\left(\exp\left(\frac{1}{4} \left(\Phi''_{0,x\overline{x}}\right)^{-1} \partial_x \cdot \partial_{\overline{x}}\right)e^Q\right)(x) \\ =
C\, \exp\left({\rm vc}_{y,\theta}\left(-4\Psi_{{\rm herm}}(x-y,\overline{x}-\theta) + Q^{\pi}(y,\theta)\right)\right),
\end{multline}
for some $C\neq 0$.

\bigskip
\noindent
Let $G$ be a holomorphic quadratic form on $\comp^{2n}$, such that
$$
{\rm Im}\, G\left(x,\frac{2}{i}\frac{\partial \Phi_0}{\partial x}(x)\right) \geq 0, \quad x\in \comp^n,
$$
and such that the fundamental matrix of $G$ does not have the eigenvalues $\pm 1$. When proving Theorem \ref{theo_main_2}, we would like to give a general criterion for when an operator of the form ${\rm Op}^w(e^{iG})$ is a non-vanishing multiple of an operator the form ${\rm Top}(e^{Q})$, where $Q$ is quadratic. In view of (\ref{eq4.17.5}), this holds precisely when the polarization
\begeq
\label{eq4.17.6}
iG\left(x,\frac{2}{i}\frac{\partial \Psi_0}{\partial x}(x,z)\right), \quad (x,z)\in \comp^{2n}_{x,z}
\endeq
of the quadratic form $\displaystyle iG\left(x,\frac{2}{i}\frac{\partial \Phi_0}{\partial x}(x)\right)$ is of the form
\begeq
\label{eq4.17.7}
iG\left(x,\frac{2}{i}\frac{\partial \Psi_0}{\partial x}(x,z)\right) = {\rm vc}_{y,\theta}\left(-4\Psi_{{\rm herm}}(x-y,z-\theta) + Q^{\pi}(y,\theta)\right).
\endeq
We need to invert the critical value expression (\ref{eq4.17.7}), and to this end we shall make use of the following well known result, see~\cite[Chapter 4]{Sj82},~\cite{CGHS}.
\begin{prop}
\label{inversion}
Let $f(X,Y)$ be a holomorphic quadratic form on $\comp^{N}_X\times \comp^N_Y$ such that ${\rm det}\, f''_{XY}(X,Y) \neq 0$ and let $g(Y)$ be a holomorphic quadratic form on $\comp^N$ such that $f(0,Y)+g(Y)$ is non-degenerate. Set
\begeq
\label{eq4.18}
h(X) = {\rm vc}_Y \left(f(X,Y) + g(Y)\right),\quad X\in \comp^N,
\endeq
where the critical value is attained at a unique critical point which is non-degenerate. Then $-f(X,0) +  h(X)$ is non-degenerate, and we have the inversion formula,
\begeq
\label{eq4.19}
g(Y) = {\rm vc}_X \left(-f(X,Y) + h(X)\right).
\endeq
\end{prop}
\begin{proof}
We recall the main ideas of the proof for completeness. Let us introduce the complex linear canonical transformation
\begeq
\label{eq4.20}
\kappa: \comp^{2N} \ni (Y,-f'_Y(X,Y)) \mapsto (X,f'_X(X,Y))\in \comp^{2N}
\endeq
and the $\comp$-Lagrangian plane
\begeq
\label{eq4.21}
\Lambda_{g} = \left\{\left(Y,g'(Y)\right),\, Y\in \comp^N\right\} \subset \comp^{2N}.
\endeq
The holomorphic quadratic form $f(0,Y) + g(Y)$ is non-degenerate on $\comp^N$ precisely when $\Lambda_g$ and $\kappa^{-1}\left(T^*_0 \comp^N\right)$ are transversal, and it follows that the $\comp$-Lagrangian plane $\kappa(\Lambda_g)$ is of the form
$$
\Lambda_h = \left\{\left(Y,h'(Y)\right),\, Y\in \comp^N\right\} \subset \comp^{2N},
$$
where the holomorphic quadratic form $h$ is given by (\ref{eq4.18}). Next, the $\comp$-Lagrangian planes $\Lambda_h$ and $\kappa(T^*_0\comp^N)$ are transversal, so that $-f(X,0) + h(X)$ is non-degenerate. Writing $\Lambda_g = \kappa^{-1}\left(\Lambda_h\right)$ and using that $\kappa^{-1}$ is of the form
\begeq
\label{eq4.22}
\kappa^{-1}: \comp^{2N} \ni (X,f'_X(X,Y)) \mapsto (Y,-f'_Y(X,Y))\in \comp^{2N},
\endeq
we infer therefore the inversion formula (\ref{eq4.19}).
\end{proof}

\medskip
\noindent
When applying Proposition \ref{inversion} to (\ref{eq4.17.7}), we let $N = 2n$, $Y = (x,z)\in \comp^{2n}$, $X = (y,\theta) \in \comp^{2n}$,
$$
g(Y) = iG\left(x,\frac{2}{i}\frac{\partial \Psi_0}{\partial x}(x,z)\right),
$$
and
\begeq
\label{eq4.23}
f(X,Y) = 4\Psi_{{\rm herm}}(x-y,z-\theta) = 4\Phi''_{0,\overline{x}x}(y-x)\cdot (\theta -z).
\endeq
We have
\begeq
\label{eq4.24}
f''_{XY} = f''_{(y,\theta),(x,z)} = \left( \begin{array}{ccc}
f''_{yx} & f''_{yz} \\\
f''_{\theta x} & f''_{\theta z}
\end{array} \right) = \left( \begin{array}{ccc}
0 & -4\Phi''_{0,x\overline{x}} \\\
-4\Phi''_{0,\overline{x} x} & 0
\end{array} \right)
\endeq
is invertible, and therefore Proposition \ref{inversion} applies. The proof of Theorem \ref{theo_main_2} is complete.

\medskip
\noindent
{\it Remark}. In the discussion above, we have considered links between the Weyl and Toeplitz quantizations in the complex domain. The purpose of this remark is to observe that such links become more direct when considering the anti-classical, rather than Weyl, quantization. Indeed, let us assume that for simplicity that the pluriharmonic part of $\Phi_0$ vanishes, so that $\Phi_0(x) = \Phi''_{0,x\overline{x}}\overline{x}\cdot x$. Given $p\in L^{\infty}(\Lambda_{\Phi_0})$, let us consider the anti-classical quantization of $p$,
\begeq
\label{eq4.25}
{\rm Op}_0(p)u(x) = \frac{1}{(2\pi)^n}\int\!\!\!\int_{\Gamma_0} e^{i(x-y)\cdot \theta} p(y,\theta) u(y)\, dy\,\wedge d\theta,
\endeq
where the contour of integration $\Gamma_0 \subset \comp^{2n}_{y,\theta}$ is given by
\begeq
\label{eq4.26}
\theta = \frac{2}{i}\frac{\partial \Phi_0}{\partial y}(y) = \frac{2}{i} \Phi''_{0,y\overline{y}}\overline{y}.
\endeq
Along $\Gamma_0$, we have $dy\wedge d\theta = 2^{2n} {\rm det}(\Phi''_{0,y\overline{y}})L(dy)$, provided that the orientation has been chosen properly. It follows that
\begin{multline}
\label{eq4.27}
{\rm Op}_0(p)u(x) = \frac{2^n {\rm det}(\Phi''_{0,y\overline{y}})}{\pi^n} \int e^{2(x-y)\cdot \Phi''_{0,y\overline{y}}\overline{y}}
p\left(y,\frac{2}{i} \Phi''_{0,y\overline{y}}\overline{y}\right) u(y)\, L(dy) \\
= \frac{2^n {\rm det}(\Phi''_{0,y\overline{y}})}{\pi^n} \int e^{2\Psi_0(x,\overline{y})} p\left(y,\frac{2}{i} \Phi''_{0,y\overline{y}}\overline{y}\right) u(y) e^{-2\Phi_0(y)}\, L(dy) = {\rm Top}(p|_{\Lambda_{\Phi_0}})u(x).
\end{multline}

\section{Example: composing special metaplectic Toeplitz operators}
\label{sec_Examples}
\setcounter{equation}{0}
The purpose of this section is to illustrate Theorem \ref{theo_main_1} and Theorem \ref{theo_main_2}, by applying them to an explicit class of metaplectic Toeplitz operators on a model Bargmann space $H_{\Phi_0}(\comp^n)$. It will be assumed throughout this section that
\begeq
\label{eq5.1}
\Phi_0(x) = \frac{\abs{x}^2}{4},\quad x\in \comp^n.
\endeq

\medskip
\noindent
Let
\begeq
\label{eq5.23}
q(x) = \lambda\abs{x}^2, \quad \widetilde{q}(x) = \widetilde{\lambda}\abs{x}^2,\quad x\in \comp^n,
\endeq
where $\lambda, \widetilde{\lambda}\in \comp$ satisfy $\displaystyle {\rm Re}\, \lambda < \frac{1}{4}$, $\displaystyle {\rm Re}\, \widetilde{\lambda} < \frac{1}{4}$, so that the assumptions (\ref{eq1.3}), (\ref{eq1.9}) hold. The Weyl symbol $a$ of the operator ${\rm Top}(e^q)$ has been computed in~\cite[Section 4]{CoHiSjWh} by evaluating the Gaussian integral (\ref{eq4.11.5}), and we recall from that work that it is given by
\begeq
\label{eq5.24}
a\left(x,\frac{2}{i}\frac{\partial \Phi_0}{\partial x}(x)\right) = C\, \exp\left(\frac{\lambda}{1-\lambda}\abs{x}^2\right),\quad x\in \comp^n, \quad C\neq 0.
\endeq
Here we notice that
\begeq
\label{eq5.24.1}
{\rm Re}\, \left(\frac{\lambda}{1 - \lambda}\right) = \frac{1 - \abs{1-2\lambda}^2}{4 \abs{1-\lambda}^2}.
\endeq
Similarly, the Weyl symbol $\widetilde{a}$ of the operator ${\rm Top}(e^{\widetilde{q}})$ has the form
\begeq
\label{eq5.25}
\widetilde{a}\left(x,\frac{2}{i}\frac{\partial \Phi_0}{\partial x}(x)\right) = \widetilde{C}\, \exp\left(\frac{\widetilde{\lambda}}{1-\widetilde{\lambda}}\abs{x}^2\right),\quad x\in \comp^n, \quad \widetilde{C} \neq 0.
\endeq
Following (\ref{eq1.11}), we shall assume that
\begeq
\label{eq5.26}
a\in L^{\infty}(\Lambda_{\Phi_0}),\quad \widetilde{a} \in L^{\infty}(\Lambda_{\Phi_0}),
\endeq
which, in view of (\ref{eq5.24.1}), is equivalent to the conditions
\begeq
\label{eq5.27}
\abs{1-2\lambda}\geq 1, \quad \abs{1-2\widetilde{\lambda}}\geq 1,
\endeq
respectively.

\bigskip
\noindent
Using (\ref{eq5.1}), (\ref{eq5.24}), and (\ref{eq5.25}), we get next
\begeq
\label{eq5.28}
a(x,\xi) = C\, \exp\left(iF(x,\xi)\right),\quad F(x,\xi) = \frac{2\lambda}{1-\lambda} x\cdot \xi, \quad (x,\xi)\in \comp^{2n},
\endeq
\begeq
\label{eq5.29}
\widetilde{a}(x,\xi) = \widetilde{C}\, \exp\left(i\widetilde{F}(x,\xi)\right),\quad \widetilde{F}(x,\xi) = \frac{2\widetilde{\lambda}}{1-\widetilde{\lambda}} x\cdot \xi, \quad (x,\xi) \in \comp^{2n},
\endeq
and recalling (\ref{eq1.8}), we see that the fundamental matrices ${\cal F}$, $\widetilde{\cal F}$ of the quadratic forms $F$, $\widetilde{F}$, respectively, are given by
\begeq
\label{eq5.30}
{\cal F} = \frac{\lambda}{1-\lambda} \begin{pmatrix}1 &0\\ 0 &-1\end{pmatrix}, \quad \widetilde{{\cal F}} = \frac{\widetilde{\lambda}}{1-\widetilde{\lambda}} \begin{pmatrix}1 &0\\ 0 &-1\end{pmatrix}.
\endeq
In particular, $\pm 1 \notin {\rm Spec}({\cal F})$, $\pm 1 \notin {\rm Spec}(\widetilde{\cal F})$. In order to apply Theorem \ref{theo_main_1}, we should also check that $-1\notin {\rm Spec}(\widetilde{{\cal F}}{\cal F})$, where the product $\widetilde{{\cal F}}{\cal F}$ is of the form
\begeq
\label{eq5.31}
\widetilde{{\cal F}}{\cal F} = \frac{\widetilde{\lambda}}{1-\widetilde{\lambda}} \frac{\lambda}{1-\lambda} \begin{pmatrix}1 &0\\ 0 &1\end{pmatrix}.
\endeq
When verifying that $-1$ is not an eigenvalue of $\widetilde{{\cal F}}{\cal F}$, we observe that this is the case provided that at least one of the inequalities in (\ref{eq5.27}) is strict. Indeed, it follows from the remark following Proposition \ref{Weyl} that the strict inequality
$\abs{1-2\lambda} > 1$, say, implies that the complex linear canonical transformation associated to the Fourier integral operator ${\rm Top}(e^q)$ is strictly positive relative to $\Lambda_{\Phi_0}$, and the general arguments of Section \ref{sec_Weyl} imply then that $-1\notin {\rm Spec}(\widetilde{{\cal F}}{\cal F})$. To discuss the remaining case, it suffices to make the following elementary observation.

\begin{lemma}
\label{lemma_5.4}
Let $\lambda,\widetilde{\lambda} \in \comp$ be such that $\displaystyle {\rm Re}\, \lambda < \frac{1}{4}$,
$\displaystyle {\rm Re}\, \widetilde{\lambda} < \frac{1}{4}$, and assume that
\begeq
\label{eq5.32}
\abs{1-2\lambda}= 1, \quad \abs{1-2\widetilde{\lambda}} = 1.
\endeq
Then we have
\begeq
\label{eq5.33}
\frac{\widetilde{\lambda}}{1-\widetilde{\lambda}}\frac{\lambda}{1-\lambda} \neq -1.
\endeq
\end{lemma}
\begin{proof}
It follows from (\ref{eq5.32}) that we have
\begeq
\label{eq5.34}
{\rm Re}\, \lambda = \abs{\lambda}^2, \quad {\rm Re}\, \widetilde{\lambda} = \abs{\widetilde{\lambda}}^2,
\endeq
and therefore,
\begeq
\label{eq5.35}
\abs{1-\lambda}^2 = 1 - {\rm Re}\, \lambda, \quad \abs{1-\widetilde{\lambda}}^2 = 1 - {\rm Re}\, \widetilde{\lambda},
\endeq
implying that
\begeq
\label{eq5.36}
\frac{\lambda}{1-\lambda} = \frac{i{\rm Im}\, \lambda}{1-{\rm Re}\, \lambda}, \quad \frac{\widetilde{\lambda}}{1-\widetilde{\lambda}} = \frac{i{\rm Im}\, \widetilde{\lambda}}{1-{\rm Re}\, \widetilde{\lambda}}.
\endeq
It suffices to check that
$$
\left(\frac{\widetilde{\lambda}}{1-\widetilde{\lambda}}\frac{\lambda}{1-\lambda}\right)^2 = \frac{({\rm Im}\, \lambda)^2 ({\rm Im}\, \widetilde{\lambda})^2}{(1-{\rm Re}\, \lambda)^2 (1-{\rm Re}\, \widetilde{\lambda})^2} \neq 1,
$$
and to this end, we observe that (\ref{eq5.34}) gives
\begeq
\label{eq5.37}
\frac{({\rm Im}\, \lambda)^2 ({\rm Im}\, \widetilde{\lambda})^2}{(1-{\rm Re}\, \lambda)^2 (1-{\rm Re}\, \widetilde{\lambda})^2} =
\frac{{\rm Re}\, \lambda\, {\rm Re}\, \widetilde{\lambda}}{(1-{\rm Re}\, \lambda) (1-{\rm Re}\, \widetilde{\lambda})} \neq 1,
\endeq
since ${\rm Re}\, \lambda + {\rm Re}\, \widetilde{\lambda} < 1$. The proof is complete.
\end{proof}

\bigskip
\noindent
An application of Theorem~\ref{theo_main_1} gives therefore that
\begeq
\label{eq5.38}
{\rm Top}(e^{\widetilde{q}}) \circ {\rm Top}(e^q) = C\, {\rm Op}^w(e^{i\widehat{F}}),\quad C\neq 0,
\endeq
where the fundamental matrix $\widehat{\cal F}$ of the quadratic form $\widehat{F}$ is given by
\begin{multline}
\label{eq5.39}
\widehat{\cal F} = (1 + {\cal F}) (1 + \widetilde{\cal F}{\cal F})^{-1} (1 + \widetilde{\cal F})-1 \\
= (1 + \widetilde{\cal F}{\cal F})^{-1} (1 + \widetilde{\cal F})(1 + {\cal F}) - 1 = (1 + \widetilde{\cal F}{\cal F})^{-1}({\cal F} + \widetilde{\cal F}).
\end{multline}
Here we have used (\ref{eq1.14}) as well as the fact that the matrices ${\cal F}$, $\widetilde{\cal F}$ commute, in view of (\ref{eq5.30}). A simple computation using (\ref{eq5.30}), (\ref{eq5.31}), and (\ref{eq5.39}) gives that
\begeq
\label{eq5.40}
\widehat{\cal F} = \frac{\lambda + \widetilde{\lambda} - 2\lambda \widetilde{\lambda}}{1 - \lambda - \widetilde{\lambda} + 2\lambda \widetilde{\lambda}} \begin{pmatrix}1 &0\\ 0 &-1\end{pmatrix}.
\endeq

\medskip
\noindent
The computations above may be summarized in the following result.

\begin{prop}
\label{prop_example}
Let $\displaystyle \Phi_0(x) = \frac{\abs{x}^2}{4}$, and let $q(x) = \lambda \abs{x}^2$, $\widetilde{q}(x) = \widetilde{\lambda}\abs{x}^2$,
with $\lambda, \widetilde{\lambda} \in \comp$ such that $\displaystyle {\rm Re}\, \lambda < \frac{1}{4}$, $\displaystyle {\rm Re}\, \widetilde{\lambda} < \frac{1}{4}$. Assume that
\begeq
\label{eq5.40.1}
\abs{1-2\lambda} \geq 1, \quad \abs{1-2\widetilde{\lambda}}\geq 1.
\endeq
We have
\begeq
\label{eq5.41}
{\rm Top}(e^{\widetilde{q}}) \circ {\rm Top}(e^q) = C\, {\rm Op}^w(e^{i\widehat{F}}): H_{\Phi_0}(\comp^n) \rightarrow H_{\Phi_0}(\comp^n),
\endeq
for some $C\neq 0$. Here the holomorphic quadratic form $\widehat{F}$ is given by
\begeq
\label{eq5.42}
\widehat{F}(x,\xi) = \frac{2(\lambda + \widetilde{\lambda} - 2\lambda \widetilde{\lambda})}{1 - \lambda - \widetilde{\lambda} + 2\lambda \widetilde{\lambda}}\, x\cdot \xi, \quad (x,\xi)\in \comp^{2n}.
\endeq
We have
$$
{\rm Im}\, \widehat{F}\left(x,\frac{2}{i}\frac{\partial \Phi_0}{\partial x}(x)\right)\geq 0,\quad x\in \comp^n,
$$
and the fundamental matrix of $\widehat{F}$ does not have the eigenvalues $\pm 1$.
\end{prop}

\bigskip
\noindent
We shall next apply Theorem \ref{theo_main_2} to the Weyl quantization in (\ref{eq5.41}). To this end, we observe that the holomorphic quadratic form
\begin{multline}
\label{eq5.43}
i \widehat{F}\left(x,\frac{2}{i}\frac{\partial \Psi_0}{\partial x}(x,z)\right) + 4 \Psi_{{\rm herm}}(x,z)
= i\frac{2(\lambda + \widetilde{\lambda} - 2\lambda \widetilde{\lambda})}{1 - \lambda - \widetilde{\lambda} + 2\lambda \widetilde{\lambda}}\,
x\cdot \frac{2}{i}\frac{z}{4} + x\cdot z \\ = \frac{1}{1 - \lambda - \widetilde{\lambda} + 2\lambda \widetilde{\lambda}}\, x\cdot z
\end{multline}
is non-degenerate on $\comp^{2n}_{x,z}$, and following (\ref{eq1.15.04}), let us set
\begeq
\label{eq5.44}
Q^{\pi}(y,\theta) = {\rm vc}_{x,z} \left((x-y)\cdot (z-\theta) + \frac{\lambda + \widetilde{\lambda} - 2\lambda \widetilde{\lambda}}{1 - \lambda - \widetilde{\lambda} + 2\lambda \widetilde{\lambda}}\, x\cdot z\right).
\endeq
We obtain after a straightforward computation that
\begeq
\label{eq5.45}
Q^{\pi}(y,\theta) = \left(\lambda + \widetilde{\lambda} - 2\lambda \widetilde{\lambda}\right) y\cdot \theta,
\endeq
and an application of Theorem \ref{theo_main_2} gives us therefore the following result.

\begin{prop}
\label{prop_example1}
Let us make the same assumptions as in Proposition {\rm \ref{prop_example}} and assume furthermore that
\begeq
\label{eq5.46}
{\rm Re}\, \left(\lambda + \widetilde{\lambda} - 2\lambda \widetilde{\lambda}\right) < \frac{1}{4}.
\endeq
Then we have
\begeq
\label{eq5.47}
{\rm Top}(e^{\widetilde{q}}) \circ {\rm Top}(e^q) = C\, {\rm Top}(e^{\widehat{q}}): H_{\Phi_0}(\comp^n) \rightarrow H_{\Phi_0}(\comp^n),
\endeq
for some constant $C\neq 0$, where $\widehat{q}(x) = \left(\lambda + \widetilde{\lambda} - 2\lambda \widetilde{\lambda}\right) \abs{x}^2$.
\end{prop}

\bigskip
\noindent
{\it Remark}. In the case when $\lambda = i$, $\widetilde{\lambda} = -i$, the result of Proposition \ref{prop_example1} has been observed in~\cite{BCH}.

\medskip
\noindent
{\it Remark}. Let us set, following~\cite{LC01},
\begeq
\label{eq5.48}
\lambda = \widetilde{\lambda} = \frac{1 + 2i}{5},
\endeq
so that (\ref{eq5.40.1}) holds, with the equality sign. Proposition \ref{prop_example} applies in this case, and we find that
the quadratic form $\widehat{F}$ given in (\ref{eq5.42}) satisfies
$$
{\rm Im}\, \widehat{F}\left(x,\frac{2}{i}\frac{\partial \Phi_0}{\partial x}(x)\right) = 0,\quad x\in \comp^n.
$$
It has been established in~\cite{LC01} that the composition in (\ref{eq5.41}) satisfies
\begeq
\label{eq5.49}
\norm{\left({\rm Top}(e^q)\right)^2 - {\rm Top}(p)}_{{\cal L}(H_{\Phi_0}({\bf C}^n), H_{\Phi_0}({\bf C}^n))} \geq 1,
\endeq
for all $p: \comp^n \rightarrow \comp$ measurable such that
\begeq
\label{eq5.50}
e^{2\Psi_0(\cdot,\overline{y})} \in {\cal D}({\rm Top}(p)),\quad y\in \comp^n.
\endeq
Notice that in this case we have
\begeq
\label{eq5.51}
{\rm Re}\, \left(2\lambda - 2\lambda^2\right) = \frac{16}{25} > \frac{1}{4},
\endeq
and therefore the assumption (\ref{eq1.15.1}) in Theorem \ref{theo_main_2} cannot be removed entirely.

\bigskip
\noindent
{\it Remark}. The purpose of this remark is to discuss the composition of more general metaplectic Toeplitz operators of the form ${\rm Top}(e^q)$, considered in~\cite[Theorem 4.1]{CoHiSj23}. Thus, let
\begeq
\label{eq5.52}
q(x) = \lambda\abs{x}^2 + A\overline{x}\cdot \overline{x}, \quad \widetilde{q}(x) = \widetilde{\lambda}\abs{x}^2 + \widetilde{A}\overline{x}\cdot \overline{x}, \quad x\in \comp^n,
\endeq
where $\lambda$, $\widetilde{\lambda} \in \comp$ and $A$, $\widetilde{A}$ are $n\times n$ complex symmetric matrices, such that
\begeq
\label{eq5.53}
{\rm Re}\, \lambda + \norm{A} < \frac{1}{4}, \quad {\rm Re}\, \widetilde{\lambda} + \norm{\widetilde{A}} < \frac{1}{4},
\endeq
and
\begeq
\label{eq5.54}
4\norm{A} \leq \frac{1-\abs{\gamma}^2}{\abs{\gamma}^2}, \quad 4\norm{\widetilde{A}} \leq \frac{1-\abs{\widetilde{\gamma}}^2}{\abs{\widetilde{\gamma}}^2}.
\endeq
Here the norm is Euclidean and
\begeq
\label{eq5.55}
\gamma = \frac{1}{1-2\lambda}, \quad \widetilde{\gamma} = \frac{1}{1-2\widetilde{\lambda}}.
\endeq
It has been established in~\cite[Theorem 4.1]{CoHiSj23} that the conditions (\ref{eq5.54}) are equivalent to the boundedness of the operators
${\rm Top}(e^q)$, ${\rm Top}(e^{\widetilde{q}})$, respectively, on the Bargmann space $H_{\Phi_0}(\comp^n)$.

\medskip
\noindent
When computing the bounded operator ${\rm Top}(e^{\widetilde{q}})\circ {\rm Top}(e^q)$, rather than applying Theorem \ref{theo_main_1}, following~\cite{CoHiSjWh},~\cite{CoHiSj23}, we shall consider the composition acting directly on the space of coherent states given by
\begeq
\label{eq5.55.1}
k_w(x) = C_{\Phi_0} e^{2\Psi_0(x,\overline{w}) - \Phi_0(w)},\quad w\in \comp^n.
\endeq
Here the constant $C_{\Phi_0} > 0$ is chosen suitably so that $\norm{k_w}_{H_{\Phi_0}({\bf C}^n)} = 1$, $w\in \comp^n$. A straightforward computation making use of~\cite[equations (2.32), (2.33)]{CoHiSj23}, or alternatively, of~\cite[equations (4.13), (4.14), (4.15)]{CoHiSjWh} shows  that
\begeq
\label{eq5.56}
\left({\rm Top}(e^{\widetilde{q}})\circ {\rm Top}(e^q)e^{2\Psi_0(\cdot,\overline{w})}\right)(x) = (\widetilde{\gamma}\gamma)^n e^{2\Psi_0(x,\widetilde{\gamma}\gamma\overline{w})} \exp\left(\widetilde{A}\widetilde{\gamma}\gamma\overline{w}\cdot \widetilde{\gamma}\gamma\overline{w} + A\gamma \overline{w}\cdot \gamma \overline{w}\right).
\endeq
Here we have also used the following more precise version of~\cite[equation (4.13)]{CoHiSjWh},
\begeq
\label{eq5.57}
\left({\rm Top}(e^{\lambda\abs{x}^2})e^{2\Psi_0(\cdot,\overline{w})}\right)(x) = \gamma^n e^{2\Psi_0(x,\gamma \overline{w})},
\endeq
which follows by the exact stationary phase. Setting
\begeq
\label{eq5.58}
\widehat{\gamma} = \widetilde{\gamma}\gamma, \quad \widehat{A} = \widetilde{A} + \frac{1}{\widetilde{\gamma}^2}A,
\endeq
we obtain from (\ref{eq5.56}) that
\begeq
\label{eq5.59}
\left({\rm Top}(e^{\widetilde{q}})\circ {\rm Top}(e^q)k_w\right)(x) = C_{\Phi_0} \widehat{\gamma}^n e^{2\Psi_0(x,\widehat{\gamma}\overline{w}) - \Phi_0(w)}\,\exp\left(\widehat{A}\widehat{\gamma}\overline{w}\cdot \widehat{\gamma}\overline{w}\right).
\endeq
Let us now set
\begeq
\label{eq5.60}
\widehat{q}(x) = \widehat{\lambda}\abs{x}^2 + \widehat{A}\overline{x}\cdot \overline{x},\quad \widehat{\lambda} = \lambda + \widetilde{\lambda} - 2\lambda \widetilde{\lambda},
\endeq
so that
$$
\widehat{\gamma} = \frac{1}{1-2\widehat{\lambda}}.
$$
Assuming that
\begeq
\label{eq5.61}
{\rm Re}\, \widehat{\lambda} + \norm{\widehat{A}} < \frac{1}{4},
\endeq
so that the Toeplitz operator ${\rm Top}(e^{\widehat{q}})$ is densely defined, we conclude, in view of (\ref{eq5.59}), that the following identity holds on the common dense domain given by the linear span of the coherent states $k_w$, $w\in \comp^n$.
\begeq
\label{eq5.62}
{\rm Top}(e^{\widetilde{q}})\circ {\rm Top}(e^q) = {\rm Top}(e^{\widehat{q}}).
\endeq
Here we observe that in view of (\ref{eq5.54}), we have
\begeq
\label{eq5.63}
4\norm{\widehat{A}} \leq 4\norm{\widetilde{A}} + \frac{4}{\abs{\widetilde{\gamma}}^2} \norm{A} \leq  \frac{1-\abs{\widetilde{\gamma}}^2}{\abs{\widetilde{\gamma}}^2} + \frac{1-\abs{\gamma}^2}{\abs{\widetilde{\gamma}}^2\abs{\gamma}^2}
= \frac{1-\abs{\widehat{\gamma}^2}}{\abs{\widehat{\gamma}^2}},
\endeq
and hence the operator ${\rm Top}(e^{\widehat{q}})$ is bounded on $H_{\Phi_0}(\comp^n)$, in view of~\cite[Theorem 4.1]{CoHiSj23}. We obtain the composition result,
\begeq
\label{eq5.64}
{\rm Top}(e^{\widetilde{q}})\circ {\rm Top}(e^q) = {\rm Top}(e^{\widehat{q}}): H_{\Phi_0}(\comp^n) \rightarrow H_{\Phi_0}(\comp^n),
\endeq
provided that (\ref{eq5.53}), (\ref{eq5.54}), and (\ref{eq5.61}) hold.

\begin{appendix}
\section{Adjoints of complex FIOs}
\label{appA}
\setcounter{equation}{0}
The purpose of this appendix is to continue the discussion started in~\cite[Appendix A]{CoHiSj} and to review some of the basic facts concerning adjoints of metaplectic Fourier integral operators in the complex domain. Let $\Phi_0$ be a strictly plurisubharmonic quadratic form on $\comp^n$, and let us recall from~\cite{Sj95},~\cite{HiSj15} that the orthogonal projection
\begeq
\label{app1}
\Pi_{\Phi_0}: L^2(\comp^n, e^{-2\Phi_0}L(dx)) \rightarrow H_{\Phi_0}(\comp^n)
\endeq
is given by
\begeq
\label{app2}
\Pi_{\Phi_0}u(x) = a_0 \int e^{2(\Psi_0(x,\overline{y}) - \Phi_0(y))} u(y)\, dy\,d\overline{y},\quad a_0\neq 0.
\endeq
Here $\Psi_0$ is the polarization of $\Phi_0$. In polarized form we may write for $u$ holomorphic,
\begeq
\label{app3}
\Pi_{\Phi_0} u(x) = \iint_{\Gamma} e^{2(\Psi (x,\theta)-\Psi (y,\theta))} a_0 u(y)dy\,d\theta.
\endeq
Here the contour $\Gamma\subset \comp^{2n}$ is given by $\theta = \overline{y}$. Omitting the contour of integration, we obtain a formal factorization
\begeq
\label{app4}
\Pi_{\Phi_0} = A\circ B,
\endeq
where
\begeq
\label{app5}
Av(x) = \int e^{2\Psi_0(x,\theta)} a_0 v(\theta)\, d\theta,\quad Bu(\theta) = \int e^{-2\Psi_0(y,\theta)} u(y)\, dy.
\endeq
We notice that from this point of view, the construction of $\Pi_{\Phi_0}$ is reduced to the problem of inverting the operator $B$, i.e. finding a suitable constant amplitude $a_0\neq 0$ such that $AB=1$ in suitable $H_{{\Phi }}$-spaces. The canonical transformations associated to $A$, $B$ are given by
\begin{equation}
\label{app6}
\kappa_A = \kappa _{2\Psi_0/i}: \left(\theta,-\frac{2}{i}\partial_{\theta} \Psi_0 (x,\theta)\right)\mapsto
\left(x,\frac{2}{i}\partial_x \Psi_0(x,\theta)\right),
\end{equation}
and
\begin{equation}
\label{app7}
\kappa_B = \kappa _{2\Psi_0/i}^{-1}:\ \left(y,\frac{2}{i}\partial_y \Psi_0(y,\theta)\right)\mapsto \left(\theta,-\frac{2}{i}\partial_{\theta} \Psi_0 (y,\theta)\right),
\end{equation}
respectively. Since $\kappa_B$ is equal to the inverse of $\kappa _A$ it is clear that $AB$ is a multiple of the identity operator, and that we can choose the constant amplitude $a_0\neq 0$ in (\ref{app5}) so that $AB=1$.

\medskip
\noindent
Putting $\theta = \overline{y}$ in (\ref{app7}), we get
$$
\kappa _B:\ \left(y,\frac{2}{i}\partial _y \Phi_0(y)\right)\mapsto \left(\overline{y},-\frac{2}{i}\partial _{\overline{y}}\Phi_0
(y)\right)=\left(\overline{y},-\frac{2}{i}\partial (\Phi_0 \circ \dagger)(\overline{y})\right)\in \Lambda _{-\Phi_0 \circ \dagger}=\dagger (\Lambda_{\Phi_0}),
$$
where $\dagger $ is the operator of complex conjugation of complex numbers or elements in ${\bf C}^N$, $\dagger(z) = \overline{z}$. Thus we have the following map between two maximally totally real subspaces of $\comp^{2n}$,
$$
\kappa_B:\, \Lambda_{\Phi_0} \rightarrow \Lambda _{-\Phi_0\circ \dagger},
$$
and $\kappa _B$ in (\ref{app7}) is the holomorphic extension of this map. Here we observe that $-\Phi_0 \circ \dagger$ is strictly
pluri-{\it super}-harmonic so it would not be meaningful to say that ``$B:\, H_{\Phi_0} \to H_{-\Phi_0 \circ \dagger}$'', or that
``$A:\, H_{-\Phi_0 \circ\dagger}\to H_{\Phi_0}$.''

\bigskip
\noindent
Let $\Phi_j$, $j= 1,2$, be strictly plurisubharmonic quadratic forms on $\comp^n$, and let $\kappa: \comp^{2n} \rightarrow \comp^{2n}$ be a complex linear canonical transformation which is positive relative to $(\Lambda_{\Phi_2}, \Lambda_{\Phi_1})$, in the sense that
\begeq
\label{app7.0.1}
\frac{1}{i} \biggl(\sigma(\kappa(\rho), \iota_{\Phi_2} \kappa(\rho)) - \sigma(\rho, \iota_{\Phi_1}(\rho))\biggr) \geq 0,\quad \rho \in \comp^{2n}.
\endeq
Here, as above, $\iota_{\Phi_j}: \comp^{2n} \rightarrow \comp^{2n}$ is the unique anti-linear involution which is equal to the identity on the maximally totally real subspace $\Lambda_{\Phi_j} \subset \comp^{2n}$, $j =1,2$.

\medskip
\noindent
Let $A$ be a metaplectic Fourier integral operator quantizing $\kappa$, and let us recall from~\cite{CoHiSj} that we can realize $A$ as a linear continuous map
\begeq
\label{app7.1}
A: H_{\Phi_1}(\comp^{n}) \rightarrow H_{\Phi_2}(\comp^n).
\endeq
We shall then also write $\kappa = \kappa_A$. With $\Psi_1 = \Phi_1^{\pi}$ being the polarization of $\Phi_1$, let
\begeq
\label{app8}
\Pi_1 u(x) = a_1 \int\!\!\!\int e^{2\Psi_1(x,\overline{y})}u(y)e^{-2\Phi_1(y)} \frac{dy\, d\overline{y}}{(2i)^n}, \quad a_1 \neq 0,
\endeq
be the Bergman projection: $L^2(\comp^n,e^{-2\Phi_1}L(dx)) \to H_{\Phi_1}(\comp^n)$. Here we observe that the $(n,n)$--form
$$
\frac{dy\, d\overline{y}}{(2i)^n} = \frac{1}{(2i)^n} dy_1 \wedge \ldots\, \wedge dy_n \wedge d\overline{y}_1 \wedge \ldots\, \wedge d\overline{y}_n
$$
can be naturally identified with the Lebesgue volume form $L(dy)$ on $\comp^n$. An application of~\cite[Theorem A.1]{CoHiSj} allows us to write for $u\in H_{\Phi_1}(\comp^n)$,
\begeq
\label{app9}
Au(x) = \int\!\!\!\int K_A(x,\overline{y}) u(y)e^{-2\Phi _1(y)}\frac{dy\,d\overline{y}}{(2i)^n}.
\endeq
Here the kernel $K_A(x,\theta)$ is holomorphic on $\comp^n_x \times \comp^n_{\theta}$, with $y\mapsto \overline{K_A(x,\overline{y})} \in H_{\Phi_1}(\comp^n)$, uniquely determined by (\ref{app9}). Arguing as in~\cite{CoHiSj} we see, using the reproducing property of $\Pi_1$ on $H_{\Phi_1}(\comp^n)$, that
\begeq
\label{app10}
K_A(x,\theta) = A\left(a_1 e^{2\Psi_1(\cdot,\theta)}\right)(x),
\endeq
see also~\cite[equation (3.9)]{CoHiSj}. We infer furthermore from the discussion in~\cite[Section 3]{CoHiSj} that the kernel $K_A$ in (\ref{app10}) is of the form
\begeq
\label{app12}
K_A(x,\theta) = \widehat{a}\,e^{2\Psi(x,\theta)},
\endeq
for some $\widehat{a}\in \comp$, where $\Psi(x,\theta)$ is a holomorphic quadratic form on $\comp^n_x\times \comp^n_{\theta}$, such that ${\rm det}\, \Psi''_{x\theta} \neq 0$. Combining (\ref{app9}) and (\ref{app12}), we get therefore,
\begeq
\label{app13}
Au(x) = \widehat{a} \iint_{\Gamma} e^{2(\Psi(x,\theta) - \Psi_1(y,\theta))} u(y)\, \frac{dy\,d\theta}{(2i)^n}.
\endeq
Introducing the formal Fourier integral operator
\begeq
\label{app13.1}
\widetilde{A}v(x) = \int e^{2\Psi(x,\theta)} \widehat{a}\, v(\theta)\, d\theta,
\endeq
with the associated canonical transformation
\begeq
\label{app13.2}
\kappa_{\widetilde{A}}: \left(\theta,-\frac{2}{i}\partial_{\theta} \Psi (x,\theta)\right)\mapsto
\left(x,\frac{2}{i}\partial_x \Psi(x,\theta)\right),
\endeq
we obtain the factorization
\begeq
\label{app14}
\kappa_A = \kappa_{\widetilde{A}}\circ \kappa_{2\Psi_1/i}^{-1}.
\endeq
Here we have set, similarly to (\ref{app6}),
\begeq
\label{app15}
\kappa_{2\Psi_1/i}: \left(\theta,-\frac{2}{i}\partial_{\theta} \Psi_1 (y,\theta)\right)\mapsto
\left(y,\frac{2}{i}\partial_y \Psi_1(y,\theta)\right).
\endeq

\bigskip
\noindent
The Hilbert space adjoint $A^*: H_{\Phi_2}(\comp^n) \rightarrow H_{\Phi_1}(\comp^n)$ of $A$ in (\ref{app7.1}) satisfies
\begeq
\label{app17}
A^*v(y) = \int\!\!\!\int K_{A^*}(y,\overline{x}) v(x) e^{-2\Phi_2(x)}\, \frac{dx\,d\overline{x}}{(2i)^n},
\endeq
where $K_{A^*}(y,\overline{x}) = \overline{K_A(x,\overline{y})}$, so that
\begeq
\label{app18}
K_{A^*}(y,x) = \overline{K_A(\overline{x},\overline{y})}.
\endeq
We shall now compute the canonical transformation $\kappa_{A^*}$ associated to the Fourier integral operator $A^*$. When doing so, recalling the notation $\dagger(z) = \overline{z}$, $z\in \comp^N$, let us put $f^\dagger = \dagger\circ f\circ \dagger$, if $f$ is a continuous function on $\comp^N$. When $f$ is of class $C^1$, we have $\partial (\dagger \circ f)=\dagger\circ \overline{\partial }f$,
$\overline{\partial }(\dagger \circ f)=\dagger\circ \partial f$. We have the analogous relations for the composition with $\dagger$  to the right and it follows that
\begeq
\label{app18.1}
\partial (f^\dagger )=(\partial f)^\dagger ,\ \overline{\partial}(f^\dagger )=(\overline{\partial }f)^\dagger.
\endeq

\medskip
\noindent
From (\ref{app12}) and (\ref{app18}), we get $K_{A^*}(y,x) = \widetilde{a}\, e^{2\Psi^{\dagger}(x,y)}$, for some $\widetilde{a}\in \comp$, and combining this with (\ref{app17}), we get that
\begeq
\label{app19}
A^*v(y) = \widetilde{a}\iint_{\Gamma} e^{2(\Psi^{\dagger}(\theta,y) - \Psi_2(x,\theta))} v(x)\, \frac{dx\,d\theta}{(2i)^n}.
\endeq
Here the contour $\Gamma$ given by $\theta = \overline{x}$ and $\Psi_2 = \Phi_2^{\pi}$ is the polarization of $\Phi_2$. Associated to the formal Fourier integral operator
\begeq
\label{app20}
\widetilde{A^*}v(y) = \widetilde{a} \int e^{2\Psi^{\dagger}(\theta,y)} v(\theta)\, d\theta
\endeq
is the canonical transformation
\begeq
\label{app21}
\kappa_{\widetilde{A^*}}: \left(\theta,-\frac{2}{i}\partial_{\theta} \Psi^{\dagger}(\theta,y)\right) \mapsto \left(y,\frac{2}{i}\partial_y \Psi^{\dagger}(\theta,y)\right),
\endeq
and it follows from (\ref{app19}), similarly to (\ref{app14}), that the following factorization holds,
\begeq
\label{app22}
\kappa_{A^*} = \kappa_{\widetilde{A^*}}\circ \kappa_{2\Psi_2/i}^{-1}.
\endeq
Here $\kappa_{2\Psi_2/i}$ is defined similarly to (\ref{app15}). We shall now simplify (\ref{app21}). In view of (\ref{app18.1}) we have
\begeq
\label{app23}
\kappa_{\widetilde{A^*}}: \dagger\left(\dagger \theta,\frac{2}{i}\partial_{\theta} \Psi(\dagger \theta,\dagger y)\right) \mapsto \dagger \left(\dagger y,-\frac{2}{i}\partial_y \Psi(\dagger \theta,\dagger y)\right),
\endeq
and replacing $\dagger \theta$, $\dagger y$ by $\theta$, $y$, we get
\begeq
\label{app24}
\kappa_{\widetilde{A^*}}: \dagger\left(\theta,\frac{2}{i}\partial_{\theta} \Psi(\theta,y)\right) \mapsto \dagger \left(y,-\frac{2}{i}\partial_y \Psi(\theta,y)\right).
\endeq
Using the fact that $\dagger^{-1} =\dagger$ together with (\ref{app13.2}), (\ref{app24}) we obtain that
\begeq
\label{app25}
\dagger \circ \kappa_{\widetilde{A^*}} \circ \dagger = \kappa_{\widetilde{A}}^{-1}.
\endeq
We combine (\ref{app25}) with (\ref{app14}), (\ref{app22}), and get
$$
\dagger \circ \kappa_{A^*} \circ \kappa_{2\Psi_2/i} \circ \dagger = \kappa_{2\Psi_1/i}^{-1} \circ \kappa_A^{-1},
$$
\begeq
\label{app26}
\left(\kappa_{2\Psi_1/i}\circ \dagger\right) \circ \kappa_{A^*} \circ \left(\kappa_{2\Psi_2/i} \circ \dagger\right) = \kappa_A^{-1}.
\endeq
Here (\ref{app15}) gives that
$$
\kappa_{2\Psi_1/i} \circ \dagger: \left(\overline{y},\frac{2}{i}\overline{\partial_y \Psi_1(x,y)}\right) \mapsto \left(x, \frac{2}{i} \partial_x \Psi_1(x,y)\right),
$$
or after the substitution $y\mapsto \overline{y}$,
\begeq
\label{app27}
\kappa_{2\Psi_1/i} \circ \dagger: \left(y,\frac{2}{i}\overline{\partial_y \Psi_1(x,\overline{y})}\right) \mapsto \left(x,
\frac{2}{i} \partial_x \Psi_1(x,\overline{y})\right).
\endeq
Recalling~\cite[equation (2.4)]{CoHiSj}, we conclude that $\kappa_{2\Psi_1/i} \circ \dagger = \iota_{\Phi_1}$ is the unique antilinear involution
$$
\iota_{\Phi_1}: \comp^{2n} \rightarrow \comp^{2n},
$$
which is equal to the identity on the maximally totally real subspace $\Lambda_{\Phi_1} \subset \comp^{2n}$. Similarly,
$\kappa_{2\Psi_2/i} \circ \dagger = \iota_{\Phi_2}$ is the unique antilinear involution on $\comp^{2n}$ which agrees with the identity on $\Lambda_{\Phi_2} \subset \comp^{2n}$. We get from (\ref{app26}) that
$$
\iota_{\Phi_1} \circ \kappa_{A^*} \circ \iota_{\Phi_2} = \kappa_{A}^{-1},
$$
or in other words,
\begeq
\label{app28}
\kappa_{A^*} = \iota_{\Phi_1} \circ \kappa_A^{-1} \circ \iota_{\Phi_2}.
\endeq

\bigskip
\noindent
The discussion in the appendix can be summarized in the following result.
\begin{theo}
Let $\Phi_j$ be strictly plurisubharmonic quadratic forms on $\comp^n$, $j=1,2$, and $\kappa: \comp^{2n} \rightarrow \comp^{2n}$ be a complex linear canonical transformation which is positive relative to $(\Lambda_{\Phi_2}, \Lambda_{\Phi_1})$, in the sense that {\rm (\ref{app7.0.1})} holds. Let $A: H_{\Phi_1}(\comp^n) \rightarrow H_{\Phi_2}(\comp^n)$ be a realization of a Fourier integral operator associated to the canonical transformation $\kappa = \kappa_A$. The canonical transformation $\kappa_{A^*}$ associated to the complex adjoint $A^*: H_{\Phi_2}(\comp^n) \rightarrow H_{\Phi_1}(\comp^n)$ of $A$ is given by
\begeq
\label{app29}
\kappa_{A^*} = \iota_{\Phi_1} \circ \kappa_A^{-1} \circ \iota_{\Phi_2}.
\endeq
Here $\iota_{\Phi_j}$ is the anti-holomorphic reflection in $\Lambda_{\Phi_j}$, $j=1,2$.
\end{theo}

\medskip
\noindent
{\it Remark}. It follows from (\ref{app29}) and the positivity of $\kappa_A$ relative to $(\Lambda_{\Phi_2}, \Lambda_{\Phi_1})$ that $\kappa_{A^*}$ is positive relative to $(\Lambda_{\Phi_1},\Lambda_{\Phi_2})$,
\begeq
\label{app30}
\frac{1}{i} \biggl(\sigma(\kappa_{A^*}(\rho), \iota_{\Phi_1} \kappa_{A^*}(\rho)) - \sigma(\rho, \iota_{\Phi_2}(\rho))\biggr) \geq 0,\quad \rho \in \comp^{2n}.
\endeq

\end{appendix}

\end{document}